\documentclass[11pt,twoside]{article}
\usepackage[margin=1.3in]{geometry}

\usepackage{latexsym,amsfonts,amsmath,amsthm,amssymb,makeidx}
\usepackage[title]{appendix}
\usepackage{CJK,CJKnumb,CJKulem,times,dsfont,ifthen,mathrsfs,latexsym,amsfonts, color}
\usepackage{amsmath,amsthm,makeidx,fontenc,amssymb,bm,graphicx,psfrag,listings, curves,extarrows}
\usepackage{cite}

\let\oldbibliography\thebibliography
\renewcommand{\thebibliography}[1]{%
\oldbibliography{#1}%
\setlength{\itemsep}{0pt}%
}

\makeindex
\newtheorem{theorem}{Theorem}[section]
\newtheorem{lemma}[theorem]{Lemma} 
\newtheorem{proposition}[theorem]{Proposition}
\newtheorem{definition}[theorem]{Definition}
\newtheorem{example}[theorem]{Example}
\newtheorem{corollary}[theorem]{Corollary}
\newtheorem{remark}[theorem]{Remark}

\newcommand{\R}{\mathbb R}

\newcommand{\bt}{\begin{theorem}}
\newcommand{\et}{\end{theorem}}
\newcommand{\bl}{\begin{lemma}}
\newcommand{\el}{\end{lemma}}
\newcommand{\bd}{\begin{definition}}
\newcommand{\ed}{\end{definition}}
\newcommand{\bc}{\begin{corollary}}
\newcommand{\ec}{\end{corollary}}
\newcommand{\bp}{\begin{proof}}
\newcommand{\ep}{\end{proof}}
\newcommand{\bx}{\begin{example}}
\newcommand{\ex}{\end{example}}
\newcommand{\bi}{\begin{exercise}}
\newcommand{\ei}{\end{exercise}}
\newcommand{\bo}{\begin{prop}}
\newcommand{\eo}{\end{prop}}
\newcommand{\br}{\begin{remark}}
\newcommand{\er}{\end{remark}}
\newcommand{\be}{\begin{equation}}
\newcommand{\ee}{\end{equation}}
\newcommand{\ba}{\begin{align}}
\newcommand{\ea}{\end{align}}
\newcommand{\bn}{\begin{enumerate}}
\newcommand{\en}{\end{enumerate}}
\newcommand{\bg}{\begin{align*}}
\newcommand{\bcs}{\begin{cases}}
\newcommand{\ecs}{\end{cases}}

\newcommand{\bean}{\begin{eqnarray*}}
\newcommand{\eean}{\end{eqnarray*}}

\numberwithin{equation}{section}

\begin{document}
\title{ {\bf Local behavior of positive solutions of higher order conformally invariant equations with a singular set} }
\date{May  25,  2020} 
\author{\\ { Xusheng Du
}
 \;\;  and \;\;   Hui Yang
 }
\maketitle

\begin{center}
\begin{minipage}{130mm}
\begin{center}{\bf Abstract}\end{center}
We study some properties of positive solutions to the higher order conformally invariant equation with a singular set 
$$
(-\Delta)^m u = u^{\frac{n+2m}{n-2m}} ~~~~~~  \textmd{in} ~ \Omega \backslash \Lambda,
$$
where $\Omega \subset \mathbb{R}^n$ is an open domain, $\Lambda$ is a closed subset of $\mathbb{R}^n$,  $1 \leq m  < n/2$ and $m$ is an integer.   We first establish an asymptotic blow up rate estimate for positive solutions near the singular set $\Lambda$ when $\Lambda \subset \Omega$ is a compact set with the upper Minkowski dimension $\overline{\textmd{dim}}_M(\Lambda) < \frac{n-2m}{2}$, or is a smooth $k$-dimensional closed manifold with $k\leq \frac{n-2m}{2}$.   We also show the asymptotic symmetry of singular positive solutions suppose  $\Lambda \subset \Omega$ is a smooth $k$-dimensional closed manifold with $k\leq \frac{n-2m}{2}$.  Finally, a global symmetry result for solutions is obtained when $\Omega$ is the whole space and $\Lambda$ is a $k$-dimensional hyperplane with $k\leq \frac{n-2m}{2}$.

\vskip0.10in

\noindent{\it Key words:} higher order conformally invariant equations,  singular set,  local behavior,  symmetry,  local integral equations.

\vskip0.10in


\end{minipage}

\end{center}

\vskip0.20in

\section{Introduction and main results}

In the seminal paper \cite{CGS}, Caffarelli, Gidas and Spruck studied the local behavior of positive solutions to the conformally invariant scalar curvature equation
\begin{equation}\label{CGS-01}
-\Delta u = u^{\frac{n+2}{n-2}}
\end{equation}
in the punctured unit ball $B_1\backslash \{0\} \subset \mathbb{R}^n$, $n \geq 3$, with an isolated singularity at the origin.   More precisely, they proved that every local positive solution $u$ is asymptotically radially symmetric near $0$, that is, $u(x)=\bar{u}(|x|)(1 + O(|x|))$ as $x \to 0$  where $\bar{u}(|x|)$ is the spherical average of $u$ on the sphere $\partial B_{|x|}(0)$. Furthermore, they showed that $u$ has a precise asymptotic behavior near the isolated singularity $0$.  Subsequent to \cite{CGS},  equation \eqref{CGS-01} and related second-order Yamabe type equations with isolated singularities have attracted a lot of attention; see, for example,  \cite{HLL,HLT,KMPS,LC96,Li06,Mar} and references therein.  The importance of studying the distributional solutions of \eqref{CGS-01} and characterizing the singular set of $u$ was indicated in the classical work of Schoen and Yau \cite{Sch88,SY88} on complete locally conformally flat manifolds.  Solutions of \eqref{CGS-01} with an isolated singularity are the simplest examples of those singular distributional solutions.  In \cite{Chen-Lin95}, Chen and Lin studied a more general case when the singular set is not isolated, which is the equation \eqref{CGS-01} in $B_1 \backslash \Lambda$ with $\Lambda$ being a higher-dimensional singular set  other than a single point.  We also refer the reader to \cite{Chen-Lin97,Zhang02} for the local estimates of positive solutions near the singular set of second order conformal scalar curvature equation.

In this paper, we are interested in the local behavior of positive solutions of the higher order conformally invariant equation with a singular set $\Lambda$: 
\begin{equation}\label{H-01} 
(-\Delta)^m u = u^{\frac{n+2m}{n-2m}} ~~~~~~  \textmd{in} ~ \Omega \backslash \Lambda,
\end{equation} 
where $\Omega \subset \mathbb{R}^n$ is an open domain, $\Lambda$ is a closed subset of $\mathbb{R}^n$,  $1 \leq m  < n/2$ and $m$ is an integer.   This equation with the critical Sobolev exponent arises as the Euler-Lagrangian equations of Sobolev inequalities \cite{CLO06,Li04,Lieb} and also arises in conformal geometry.   More precisely,  let $|dx|^2$ be the Euclidean metric and consider a conformal change $g:=u^{\frac{4}{n-4}}|dx|^2$ for some positive smooth function $u$. Then the fourth order Paneitz operator with respect to the metric  $g$  satisfies 
$$
P_2^g = u^{-\frac{n+4}{n-4}} \Delta^2(u \cdot),
$$
and the  $Q$-curvature of $g$ is given by 
$$
Q_g =\frac{2}{n-4}P_2^g(1)=\frac{2}{n-4} u^{-\frac{n+4}{n-4}} \Delta^2 u. 
$$
Hence, each positive solution $u$ of \eqref{H-01} with $m=2$   induces a conformal metric $g=u^{\frac{4}{n-4}}|dx|^2$ which has  positive constant $Q$-curvature in $\Omega \backslash \Lambda$.  For an introduction to the $Q$-curvature problem see, for instance, Hang-Yang \cite{HY}.  See also Gursky-Malchiodi \cite{GM} and Hang-Yang \cite{HY16} for the recent progresses on the $Q$-curvature problem on Riemannian manifolds.

When $\Omega = \mathbb{R}^n$ and $\Lambda=\{0\}$  is an isolated singularity,   Lin \cite{Lin98} proved that all the singular positive solutions of \eqref{H-01}  are radially symmetric about $0$ for $m=2$.   Frank-K\"{o}nig \cite{Fr-K19} classified all these singular radial solutions,  called the Fowler solutions or Delaunay type solutions, using ODE analysis.  Recently,  the higher order equation \eqref{H-01} in the punctured unit ball $B_1\backslash \{0\}$ was studied by Jin and Xiong in \cite{JX19}.  In that paper, the authors showed  the asymptotic radial symmetry of singular positive solutions and their sharp blow up rate near $0$ under the sign assumptions 
\begin{equation}\label{Si-Assu}
(-\Delta)^s u \geq 0~~~~~~ \textmd{in} ~ B_1\backslash \{0\}, ~~ s=1, \dots, m-1. 
\end{equation}
In \cite{R20},   for when  $m=2$,  Ratzkin  proved  that every local solution $u$ of \eqref{H-01}  which  satisfies  \eqref{Si-Assu}  has a refined asymptotic behavior near the isolated singularity $0$  based on the classification result of Frank-K\"{o}nig \cite{Fr-K19} and a priori upper estimates of Jin-Xiong \cite{JX19}.

In this paper, we  would like to continue the previous study of Jin-Xiong \cite{JX19} on singular positive solutions of the higher order equation \eqref{H-01} in a more general case,  that is,  $\Lambda$ is a singular set rather than  an isolated point.  For the case with  a higher-dimensional singular set, the behavior of solutions is expected to be more complicated.  To state our first  result, we recall the definition of the Minkowski dimension (see, e.g., \cite{KLV13,Mattila}).   Suppose  $E\subset \mathbb{R}^n$ is a compact set, the $\lambda$-dimensional Minkowski $r$-content of $E$ is defined by 
$$
\mathcal{M}_r^\lambda(E) = \inf \left\{n r^\lambda ~ \big| ~  E \subset \bigcup_{k=1}^n B(x_k, r),  ~ x_k \in E \right\}, 
$$ 
and the upper and lower Minkowski dimensions are defined, respectively, as 
$$
\overline{\textmd{dim}}_M(E) = \inf \Big\{\lambda \geq 0  ~ \big| ~  \limsup_{r\to 0} \mathcal{M}_r^\lambda(E) =0 \Big\}, 
$$ 
$$
\underline{\textmd{dim}}_M(E) = \inf \Big\{\lambda \geq 0  ~ \big| ~  \liminf_{r\to 0} \mathcal{M}_r^\lambda(E) =0 \Big\}. 
$$
If $\overline{\textmd{dim}}_M(E) = \underline{\textmd{dim}}_M(E) $, then the common value,  denoted by $\textmd{dim}_M(E)$,  is the Minkowski dimension of $E$.  Recall also that for a compact set $E\subset \mathbb{R}^n$, we have the relation $\textmd{dim}_H (E) \leq \underline{\textmd{dim}}_M (E)  \leq \overline{\textmd{dim}}_M (E)$, where $\textmd{dim}_H(E)$ is the Hausdorff dimension of $E$.

We will use $B_r(x)$ to denote the open ball of radius $r$ in $\mathbb{R}^n$ with center $x$ and write $B_r(0)$ as $B_r$ for short.   From now on, without loss of generality, we take the domain  $\Omega=B_2$. Firstly, we derive a local estimate of a singular positive solution $u$ near its singular set $\Lambda$ of \eqref{H-01}. 

\begin{theorem}\label{Thm01}
Suppose that $1 \leq m < n/2$ and $m$ is an integer. Let $\Lambda\subset B_{1/2}$ be a compact set with the upper Minkowski dimension $\overline{\textmd{dim}}_M(\Lambda)$ (not necessarily an integer),  $\overline{\textmd{dim}}_M(\Lambda) < \frac{n-2m}{2}$,  or be a smooth $k$-dimensional closed manifold with $k\leq \frac{n-2m}{2}$.  Let $u\in C^{2m}(B_2 \backslash \Lambda)$ be a positive solution of 
\begin{equation}\label{HOE} 
(-\Delta)^m u = u^{\frac{n+2m}{n-2m}} ~~~~~~  \textmd{in} ~B_2 \backslash \Lambda. 
\end{equation} 
Suppose
\begin{equation}\label{HOE02}  
(-\Delta)^s u \geq 0 ~~~~~~  \textmd{in} ~ B_2 \backslash \Lambda,~~ s=1, \dots, m-1.  
\end{equation} 
Then there exists a constant  $C > 0$ such that
\begin{equation}\label{Est01} 
u(x) \leq C  [d(x,  \Lambda)]^{-\frac{n-2m}{2}} 
\end{equation} 
for all $x\in B_1\backslash \Lambda$, where $d(x,  \Lambda)$ is the distance  between $x$ and $\Lambda$. 
\end{theorem}

Assume that $\Lambda \subset B_{1/2}$ is a smooth $k$-dimensional  closed manifold with $k\leq \frac{n-2m}{2}$. Let $N$ be a tubular neighborhood of $\Lambda$ such  that any point of $N$ can be uniquely expressed as the sum $x+v$ where $x\in \Lambda$ and $v\in (T_x\Lambda)^\bot$, the orthogonal complement of the tangent space of $\Lambda$ at $x$. Denote $\Pi$ the orthogonal projection of $N$ onto $\Lambda$. For small $r>0$ and $z\in \Lambda$, 
\begin{equation}\label{Pi}
\Pi_r^{-1}(z) = \{y \in N ~ | ~ \Pi(y)=z, ~~ |y-z|=r \}. 
\end{equation} 
We have the following asymptotic symmetry of  solutions near the singular set $\Lambda$.  

\begin{theorem}\label{Thm02}
Suppose $1 \leq m < n/2$ and $m$ is an integer. Let  $u\in C^{2m}(B_2 \backslash \Lambda)$ be a positive solution of \eqref{HOE}. Suppose that \eqref{HOE02} holds,  and $N$, $\Lambda$ and $\Pi$ are described as above. Then,   for $x, x^\prime \in \Pi_r^{-1}(z)$, we have
\begin{equation}\label{Asy01}
u(x)=u(x^\prime) ( 1 + O(r))~~~~~ \textmd{as}~ r \to 0^+, 
\end{equation}
where $O(r)$ is uniform for all $z\in \Lambda$. 
\end{theorem}

Remark that  when  $m=2$,  the positivity of the scalar curvature of the metric $u^{\frac{4}{n-4}} |dx|^2$ implies that $-\Delta u > 0 $.  We also mention that  Gursky-Malchiodi \cite{GM}  studied the positivity of the Paneitz operator and its Green's function under the assumption that the scalar curvature is  positive.    
Since we do not use any special structure of the open ball, $B_2$ can be replaced by general domains containing $\overline{B_{1/2}}$. Also, both of the above theorems apply to $\Lambda \subset B_2$ being compact. When $\Lambda$ is a single point, Theorems \ref{Thm01} and \ref{Thm02} have  been proved in Jin-Xiong \cite{JX19}.

Now, we give a global symmetry result when $\Omega$ is the whole Euclidean space and $\Lambda$ is a lower dimensional hyperplane.  Let $\mathbb{R}^k$ be a $k$-dimensional subspace of $\mathbb{R}^n$ with $0\leq k \leq n-1$ being an integer,  where $\mathbb{R}^0$ denotes the origin $\{0\}$.  

\begin{theorem}\label{Thm03}
Suppose that $1 \leq m < n/2$ and $m$ is an integer. Let $0\leq k \leq \frac{n-2m}{2}$ and $u \in C^{2m} (\mathbb{R}^n \backslash \mathbb{R}^k)$ be a nonnegative  solution of 
\begin{equation}\label{Whole}
(-\Delta)^m u = u^{\frac{n+2m}{n-2m}} ~~~~~~ in ~ \mathbb{R}^n \backslash \mathbb{R}^k. 
\end{equation}
Suppose there exists $x_0\in \mathbb{R}^k$ such that $\limsup_{x\to x_0}u(x) =\infty$.  Then
$$
u(x^\prime, x^{\prime\prime}) = u(x^\prime, \tilde{x}^{\prime\prime}), 
$$
where $x^\prime \in \mathbb{R}^k$ and $x^{\prime\prime}, \tilde{x}^{\prime\prime} \in \mathbb{R}^{n-k}$ that $|x^{\prime\prime}| = |\tilde{x}^{\prime\prime}|$.  In particular, If $k=0$, i.e., $\mathbb{R}^k=\{0\}$,  and the origin is a non-removable singularity, then $u$ is radially symmetric and monotonically decreasing about the origin.  
\end{theorem} 

When the singular set $\mathbb{R}^k$ is removable, i.e., $u$ can be extended as a positive smooth solution in the whole $\mathbb{R}^n$, the classification of positive solutions of \eqref{Whole} was obtained by Caffarelli-Gidas-Spruck \cite{CGS} for  $m=1$, by Lin \cite{Lin98} for $m=2$ and  by Wei-Xu \cite{Wei-Xu} for $m\geq 3$. We may also see Chen-Li-Ou\cite{CLO06} and Y.Y. Li \cite{Li04} for the classification of positive smooth solutions of  conformally invariant integral equations.  When $\mathbb{R}^k=\{0\}$ and \{0\} is a non-removable singularity, the radial symmetry of positive solutions of \eqref{Whole} was  proved by Caffarelli-Gidas-Spruck \cite{CGS} for  $m=1$ and by Lin \cite{Lin98} for $m=2$ via the method of moving planes.  Our proof is different from the ones in \cite{CGS,Lin98} and in fact,  it is easy  to  derive the following global estimate from our proof.

\begin{corollary}\label{Cor01}
Suppose that $1 \leq m < n/2$ and $m$ is an integer. Let  $u \in C^{2m} (\mathbb{R}^n \backslash \{0\})$ be a nonnegative  solution of  
\begin{equation}\label{Whole00}
(-\Delta)^m u = u^{\frac{n+2m}{n-2m}} ~~~~~~ in ~ \mathbb{R}^n \backslash \{0\}. 
\end{equation}
If the origin is a non-removable singularity, then 
there exist  two  positive constants $C_1=C_1(n, m)$  and $C_2=C_2(n, m , u)$  such that for all $x\in \mathbb{R}^n \backslash \{0\}$, 
\begin{equation}\label{dfgh}
C_2 |x|^{-\frac{n-2m}{2}} \leq u(x) \leq C_1 |x|^{-\frac{n-2m}{2}}. 
\end{equation}
\end{corollary}  

For when $m=2$, Corollary \ref{Cor01} was proved  by Frank-K\"{o}nig \cite{Fr-K19} by ODE analysis, which plays an important  role in their  classification of  all the singular positive solutions in $\mathbb{R}^n \backslash \{0\}$.   

It is well-known that the equation \eqref{HOE} is conformally invariant in the following sense. If $u$ is a solution of \eqref{HOE}, then its Kelvin transform
\begin{equation}\label{Kel}
u_{x,\lambda} (\xi) = \left( \frac{\lambda}{|\xi-x|} \right)^{n-2m} u \left( x + \frac{\lambda^2 (\xi-x)}{|\xi-x|^2}\right)
\end{equation}
is also a solution of \eqref{HOE} in the corresponding region. Recall also that in the classification of smooth positive solutions of \eqref{HOE}  in $\mathbb{R}^n$ by Lin  \cite{Lin98} and Wei-Xu \cite{Wei-Xu}, one crucial ingredient is that every entire smooth solution of \eqref{HOE} satisfies the sign conditions \eqref{HOE02} in $\mathbb{R}^n$.   This indicates that the sign conditions \eqref{HOE02}  are kept under the Kelvin transform \eqref{Kel} for entire solutions.   However, for our local equation \eqref{HOE}, the  sign conditions \eqref{HOE02} may change under the  Kelvin transform \eqref{Kel}.  
This makes it seem very difficult to prove Theorems \ref{Thm01} and \ref{Thm02}  directly using the Kelvin transform and the moving plane method  to the higher order equation \eqref{HOE},  which is the approach of Chen-Lin \cite{Chen-Lin95}  to obtain these results when $m=1$. 

We overcome this difficulty along the similar idea developed in Jin-Xiong \cite{JX19} when $\Lambda=\{0\}$, which is inspired by Jin-Li-Xiong \cite{JLX17}.   More specifically, we  rewrite the differential equation \eqref{HOE} into the local integral equation \eqref{Int} below and study the singular solutions of this integral equation.  The aim of this paper is to further develop the idea of Jin-Xiong\cite{JX19} to study the local behavior of  positive solutions to  the differential equation \eqref{HOE} with a general singular set $\Lambda$.  Moreover, we also apply this idea to study the symmetry of global singular solutions of \eqref{Whole}.

Suppose the dimension $n \geq 1$, $0< \sigma < \frac{n}{2}$ is a real number, and $\Sigma$ is a closed set in  $\mathbb{R}^n$.  We consider the local integral equation 
\begin{equation}\label{Int}
u(x) =\int_{B_2} \frac{ u(y)^{\frac{n+2\sigma}{n-2\sigma}} }{|x-y|^{n-2\sigma}} + h(x),~~~ u >0, ~~~~~ x \in B_2 \backslash \Sigma, 
\end{equation}
where $u \in L^{\frac{n+2\sigma}{n-2\sigma}}(B_2) \cap C(B_2 \backslash \Sigma)$ and $ h\in C^1(B_2)$ is a positive function.  Under the assumptions in Theorem \ref{Thm01}, one can show $u\in L_{\textmd{loc}}^{\frac{n+2\sigma}{n-2\sigma}}(B_2) $ and can rewrite the equation \eqref{HOE}  locally  into the integral equation \eqref{Int} after some scaling (see Theorem \ref{T-Lo094} in Section  2).

Next we state the corresponding results for singular solutions of the integral equation \eqref{Int}.  Denote $\mathcal{L}^n$  the $n$-dimensional Lebesgue measure on $\mathbb{R}^n$.  

\begin{theorem}\label{IEthm01}
Suppose $n \geq 1$, $0< \sigma < n/2$, and $\Sigma$ is a closed set in  $\mathbb{R}^n$ with $\mathcal{L}^n (\Sigma)=0$. Let $h\in C^1(B_2)$ be a positive function and  $u \in L^{\frac{n+2\sigma}{n-2\sigma}}(B_2) \cap C(B_2 \backslash \Sigma)$  be a positive solution of \eqref{Int}. Then there exists a constant $C>0$ such that 
\begin{equation}\label{In-Es01}
u(x) \leq C [d(x, \Sigma)]^{-\frac{n-2\sigma}{2}}
\end{equation}
for all $x\in B_1\backslash \Sigma$, where  $d(x,  \Sigma)$ is the distance  between $x$ and $\Sigma$.  
\end{theorem} 

When $0< \sigma <1$,  equation \eqref{Int} is closely related to the fractional Yamabe equation.    Fractional Yamabe equations with isolated singularities were considered in \cite{ADGW,CJSX,DPGW},  while solutions with a higher dimensional singular set have been studied by Jin-de Queiroz-Sire-Xiong \cite{JQSX}, and by Ao-Chan-DelaTorre-Fontelos-Gonzalez-Wei \cite{ACDFGW}  which develops a Mazzeo-Pacard gluing program (see \cite{MP})  in the fractional setting.  Note that for the singular set $\Sigma$ of  the integral equation \eqref{Int},  we only assume that $\Sigma$ has $n$-dimensional  Lebesgue measure 0, which is a weaker condition than the singular set of  Newton capacity 0 studied by Chen-Lin \cite{Chen-Lin95} to the second order equation \eqref{CGS-01} and the singular set of fractional capacity 0 studied by Jin-de Queiroz-Sire-Xiong \cite{JQSX} to the fractional Yamabe equation. 

Further,  if $\Sigma$ is a smooth submanifold of  $\mathbb{R}^n$,  then we also have the asymptotic symmetry of singular solutions of \eqref{Int}. 
\begin{theorem}\label{IEthm02}
Suppose $n \geq 1$, $0< \sigma < n/2$,  $\Sigma \subset \mathbb{R}^n$ is a smooth $k$-dimensional closed manifold  with $k\leq n-1$,     $N$ is a tubular neighborhood of $\Sigma$ and $\Pi$ is the orthogonal projection of $N$ onto $\Sigma$ described as before.   Let $h\in C^1(B_2)$ be a positive function and  $u \in L^{\frac{n+2\sigma}{n-2\sigma}}(B_2) \cap C(B_2 \backslash \Sigma)$  be a positive solution of \eqref{Int}. Let $A\subset B_{2}$ be a compact subset of $\Sigma$.  Then we have, for $x, x^\prime \in \Pi_r^{-1}(z)$, 
\begin{equation}\label{In-Es02}
u(x)= u(x^\prime) (1 + O(r))~~~~ \textmd{as} ~ r \to 0^+,
\end{equation}
where  $O(r)$ is uniform for all $z \in A$.  
\end{theorem} 

In Theorems \ref{IEthm01} and \ref{IEthm02}, we allow the singular set $\Sigma$ to intersect the boundary $\partial B_2$,  which is essential  for applying these results to the differential equation \eqref{HOE}. When $\Sigma=\{0\}$ is an isolated singularity,  Theorems \ref{IEthm01} and \ref{IEthm02} were proved by Jin-Xiong \cite{JX19}.

Finally, for the global singular solutions of the differential equation \eqref{Whole}, under the assumptions of Theorem \ref{Thm03}, one can also show  that $u\in L_{\textmd{loc}}^{\frac{n+2\sigma}{n-2\sigma}}(\mathbb{R}^n) $ and all the singular positive solutions  of \eqref{Whole} satisfy the following integral equation \eqref{Whole-Int} (see Theorem \ref{T-Glo}  in Section  2).    This has been proved by Chen-Li-Ou \cite{CLO06} to be true for the entire smooth positive solutions of \eqref{Whole}  in $\mathbb{R}^n$.    Their proof makes use of the sign conditions \eqref{HOE02} in $\mathbb{R}^n$ for smooth  solutions  proved by Wei-Xu \cite{Wei-Xu},  which is different from ours.  Of course, our proof also works for smooth solutions.      

Let $\mathbb{R}^k$ be a $k$-dimensional subspace of $\mathbb{R}^n$, and let $\sigma$ be a real number satisfying $0 < \sigma < n/2$.   Consider the integral equation
\begin{equation}\label{Whole-Int}
u(x)= \int_{\mathbb{R}^n} \frac{u(y)^{\frac{n+2\sigma}{n-2\sigma}}}{|x -y|^{n-2\sigma}} dy,~~~~~ x\in \mathbb{R}^n \backslash \mathbb{R}^k. 
\end{equation}
Then we have the following
\begin{theorem}\label{IEthm03}
Let $0\leq k \leq n-1$ and $u\in L_{\textmd{loc}}^{\frac{n+2\sigma}{n-2\sigma}}(\mathbb{R}^n) \cap C(\mathbb{R}^n \backslash \mathbb{R}^k)$ be a positive solution of \eqref{Whole-Int}. Suppose there exists $x_0\in \mathbb{R}^k$ such that $\limsup_{x\to x_0}u(x) =\infty$.  Then
$$
u(x^\prime, x^{\prime\prime}) = u(x^\prime, \tilde{x}^{\prime\prime}), 
$$
where $x^\prime \in \mathbb{R}^k$ and $x^{\prime\prime}, \tilde{x}^{\prime\prime} \in \mathbb{R}^{n-k}$ that $|x^{\prime\prime}| = |\tilde{x}^{\prime\prime}|$.  In particular, If $k=0$, i.e., $\mathbb{R}^k=\{0\}$,  and the origin is a non-removable singularity, then $u$ is radially symmetric and monotonically decreasing about the origin.  
\end{theorem}

When $k=0$, Theorem \ref{IEthm03} was proved by Chen-Li-Ou \cite{CLO05}.  In this case (when $k=0$), it is well-known that $c|x|^{-\frac{n-2\sigma}{2}}$ is a singular solution of \eqref{Whole-Int} for a positive constant $c$ depending only on $n$ and $\sigma$, the other singular solutions, e.g.,  the Fowler solutions,  of \eqref{Whole-Int} were  obtained by Jin-Xiong \cite{JX19}. See also \cite{GHWW,DPGW,Fr-K19} for the existence of Fowler solutions of related differential equations.

Since the integral equations \eqref{Int} and \eqref{Whole-Int} are conformally invariant (see Section \ref{S3000}),   we shall  prove Theorems \ref{IEthm01}, \ref{IEthm02} and \ref{IEthm03} using  the method of moving spheres introduced by Li-Zhu \cite{Li-Zhu} for differential equations and by Li \cite{Li04} for integral equations.  A difference from the integral equations studied in \cite{Li04,CLO06} is that our integral equation \eqref{Int} is locally defined,  and we need some more delicate estimates.   Another difference is that we are concerned with the singular solutions of \eqref{Int} and \eqref{Whole-Int}.   
More applications of the method of moving spheres can be found in \cite{CJSX,JX19,LL03,Li06,LZ03,Zhang02}.

This paper is organized as follows. In Section \ref{S2}, we show the integral representations  for  singular positive solutions to  the differential equations \eqref{HOE}  and \eqref{Whole}, respectively.  In Section \ref{S3000},  we prove the upper bounds in Theorems \ref{IEthm01} and  \ref{Thm01}.  In Section \ref{S4000}, we show the asymptotic symmetry of the solutions near the singular set  in Theorems \ref{IEthm02} and \ref{Thm02}. In Section \ref{S5000}, we show the symmetry results of global singular solutions in Theorems \ref{IEthm03} and \ref{Thm03},   where we also give the proof of Corollary \ref{Cor01}.  

\vskip0.15in  

\noindent{\bf Acknowledgments.}   Both authors  would like to thank Professor Tianling Jin for many helpful discussions and encouragement.

\section{Integral representations for  singular solutions}\label{S2}
\subsection{Local singular solutions}\label{S2-01}
In this subsection, we show that every singular positive solution of the differential equation \eqref{HOE} satisfies the integral equation \eqref{Int} in some local sense under suitable assumptions.    Firstly, we prove that under the assumptions of Theorem \ref{Thm01} (\eqref{HOE02} is not needed here),  then $u \in L_{\textmd{loc}}^{\frac{n + 2m}{n - 2m}}(B_2) $ and  $u$ is a distributional solution in the entire ball $B_2$. 

\begin{proposition}\label{P201}
Suppose that $1 \leq m < n/2$ and $m$ is an integer. Let $\Lambda\subset B_{1/2}$ be a compact set with the upper Minkowski dimension $\overline{\textmd{dim}}_M(\Lambda)$ (not necessarily an integer),  $\overline{\textmd{dim}}_M(\Lambda) < \frac{n-2m}{2}$,  or be a smooth $k$-dimensional closed manifold with $k\leq \frac{n-2m}{2}$.  Let $u\in C^{2m}(\overline{B}_2 \backslash \Lambda)$ be a positive solution of \eqref{HOE}. Then $u \in L^{\frac{n+ 2m}{n - 2m}}(B_2) $ and $u$ is a distributional solution in the entire ball $B_2$, i.e., we have 
\begin{equation}\label{Dis2}
\int_{B_2} u (-\Delta)^m \varphi  dx = \int_{B_2} u^{\frac{n+2m}{n-2m}} \varphi  dx
\end{equation}
for every  $\varphi \in C_c^\infty(B_2)$. 
\end{proposition}  
 
\begin{remark}
In a recent paper \cite{AGHW},  Ao, Gonz\'alez, Hyder and Wei  studied removable singularities and  superharmonicity of non-negative solutions to the fractional equation $(-\Delta)^\gamma u = u^{\frac{n+2\gamma}{n-2\gamma}}$ in $\mathbb{R}^n \backslash \Sigma$,  where $0< \gamma < \frac{n}{2}$.   Among other things, they proved if $\Sigma$ is a compact set in $\mathbb{R}^n$ with the upper Assouad dimension $\overline{\textmd{dim}}_A(\Sigma)   < \frac{n-2\gamma}{2}$, and  $u\in L_\gamma(\mathbb{R}^n) \cap L^{\frac{n+2\gamma}{n-2\gamma}}_{loc}(\mathbb{R}^n \backslash \Sigma)$ is a non-negative solution, then $u\in L^{\frac{n+2\gamma}{n-2\gamma}}_{loc}(\mathbb{R}^n)$ and $u$ is a distributional solution in $\mathbb{R}^n$. We also remark that for any compact set $E\subset \mathbb{R}^n$,  the relation $\overline{\textmd{dim}}_M(E) \leq \overline{\textmd{dim}}_A(E)$ holds, see, e.g., \cite{KLV13}.  
\end{remark} 
 
\begin{proof}
We first assume that $\Lambda\subset B_{1/2}$ be a compact set with the upper Minkowski dimension $\overline{\textmd{dim}}_M(\Lambda)  < \frac{n-2m}{2}$. Let $\mathcal{N}_r:=\{x\in \mathbb{R}^n ~|~\textmd{ dist}(x, \Lambda) < r\}$ be a tubular $r$-neighborhood of $\Lambda$. We fix a non-negative function $\rho\in C_c^\infty(B_1)$ satisfying $\int_{B_1} \rho dx=1$. Setting $\rho_\varepsilon(x)=\frac{1}{\varepsilon^n} \rho(\frac{x}{\varepsilon})$ for small $\varepsilon >0$.  As in Ao-Gonz\'alez-Hyder-Wei \cite{AGHW}, we define 
\begin{equation}\label{rho}
\eta_\varepsilon(x):= 1 - \int_{\mathcal{N}_{2\varepsilon}} \rho_\varepsilon(x-y) dy. 
\end{equation}
Then $\eta_\varepsilon \in C^\infty(\mathbb{R}^n)$ is a non-negative function with values in $[0,1]$,  and it satisfies
\begin{equation}\label{rho01}
\eta_\varepsilon(x) = 1 ~~~~ \textmd{on} ~\mathcal{N}_{3\varepsilon}^c~~~~~~ \textmd{and} ~~~~~~ \eta_\varepsilon(x) = 0 ~~~~ \textmd{on} ~ \mathcal{N}_{\varepsilon}. 
\end{equation} 
Moreover, we have,  for every $j=1, 2, \dots$, 
\begin{equation}\label{rho02}
|\nabla^j \eta_\varepsilon| \leq C \varepsilon^{-j}. 
\end{equation}
Next we shall use some arguments in Yang \cite{Y}. Let $\varphi_\varepsilon(x) = [\eta_\varepsilon(x)]^q$ with $q=\frac{n+2m}{2}$. Multiplying both sides of \eqref{HOE} by $\varphi_\varepsilon$ and using integration by parts, we obtain  
$$
\aligned
\int_{B_2} u^{\frac{n+2m}{n-2m}} \varphi_\varepsilon dx & = - \int_{\partial B_2} \frac{\partial(-\Delta)^{m-1} u}{\partial\nu}
dS + \int_{B_2} u (-\Delta)^m \varphi_\varepsilon dx\\
&\leq C + C \varepsilon^{-2m} \int_{\mathcal{N}_{3\varepsilon} \backslash \mathcal{N}_{\varepsilon}} u (\eta_\varepsilon)^{q-2m} dx \\
& \leq C + C \varepsilon^{-2m} \int_{\mathcal{N}_{3\varepsilon} \backslash \mathcal{N}_{\varepsilon}} u (\varphi_\varepsilon)^{\frac{n-2m}{n+2m}} dx \\
& \leq C + C \varepsilon^{-2m}  \left( \int_{B_2} u^{\frac{n+2m}{n-2m}} \varphi_\varepsilon dx \right)^{\frac{n-2m}{n+2m}} \left( \int_{\mathcal{N}_{3\varepsilon} \backslash \mathcal{N}_{\varepsilon}} 1 dx \right)^{\frac{4m}{n+2m}}. 
\endaligned
$$
Now we estimate the integral  $\int_{\mathcal{N}_{3\varepsilon} \backslash \mathcal{N}_{\varepsilon}} 1 dx$.  Choosing $\lambda > \overline{\textmd{dim}}_M(\Lambda)$ but sufficiently close to $\overline{\textmd{dim}}_M(\Lambda)$ such that $\frac{4m(n-\lambda)}{n+2m} -2m \geq 0$ (equivalently, $\lambda \leq \frac{n-2m}{2}$).  By Proposition 5.8 of \cite{KLV13}, there exist two constants $C$, $r_0>0$ such that
\begin{equation}\label{Dim09}
\mathcal{H}^{n-1} (\partial \mathcal{N}_\varepsilon) \leq C \varepsilon^{n-\lambda-1} ~~~~ \textmd{for}~ \textmd{all} ~ 0< \varepsilon <r_0,  
\end{equation}
where $\mathcal{H}^{n-1}$ is the $(n-1)$-dimensional Hausdorff measure on $\mathbb{R}^n$.   Note that the distance function $d(x):=\textmd{dist}(x, \Lambda)$ is a $1$-Lipschitz function.  It follows from Rademacher's theorem that $d(x)$ is differentiable $a.e.$ with $|\nabla d|=1$. Then by  the co-area formula (see, e.g., Evans-Gariepy \cite{EGbook}) and \eqref{Dim09},   we have
$$
\aligned
\int_{\mathcal{N}_{3\varepsilon} \backslash \mathcal{N}_{\varepsilon}} 1 dx  \leq \mathcal{L}^n (\mathcal{N}_{3\varepsilon} ) & = \int_{0}^{3\varepsilon}  \left( \int_{\partial \mathcal{N}_r} 1 d\mathcal{H}^{n-1} \right) dr \\
& \leq C \int_{0}^{3 \varepsilon} r^{n-\lambda-1} dr  \leq C \varepsilon^{n-\lambda}
\endaligned
$$
if $\varepsilon< r_0/3$.  This together with the above estimate yields 
\begin{equation}\label{908} 
\aligned
\int_{B_2} u^{\frac{n+2m}{n-2m}} \varphi_\varepsilon dx &\leq  C + C \varepsilon^{\frac{4m(n-\lambda)}{n+2m} -2m} \left( \int_{B_2} u^{\frac{n+2m}{n-2m}} \varphi_\varepsilon dx \right)^{\frac{n-2m}{n+2m}}\\
& \leq C + C  \left( \int_{B_2} u^{\frac{n+2m}{n-2m}} \varphi_\varepsilon dx \right)^{\frac{n-2m}{n+2m}},
\endaligned
\end{equation} 
where we have used the fact $\frac{4m(n-\lambda)}{n+2m} -2m \geq 0$ due to the choice of $\lambda$.  Using Young inequality in the last term in \eqref{908}, we have 
$$
\int_{B_2 \backslash \mathcal{N}_{3\varepsilon}} u^{\frac{n+2m}{n-2m}} dx \leq \int_{B_2} u^{\frac{n+2m}{n-2m}} \varphi_\varepsilon dx \leq C. 
$$
Sending $\varepsilon \to 0$, we obtain
$$
\int_{B_2}  u^{\frac{n+2m}{n-2m}} dx \leq C.
$$

When $\Lambda$ is a smooth $k$-dimensional closed manifold with $k\leq \frac{n-2m}{2}$,  the proof is very similar to the above,  and we only need to notice that in this case (when $\Lambda$  is smooth),  $\mathcal{L}^n (\mathcal{N}_{3\varepsilon}) \approx \varepsilon^{n-k}$ for $\varepsilon>0$ small.

Next, we show that $u$ is a distributional solution in $B_2$.  For any $\varphi \in C_c^{\infty}(B_2)$, using $\psi_\varepsilon:=\varphi\eta_\varepsilon$ as a test function in \eqref{HOE} with $\eta_\varepsilon$ as before gives  
\begin{equation}\label{Te567}
\int_{B_2} u^{\frac{n+2m}{n-2m}} \psi_\varepsilon dx= \int_{B_2} u\eta_\varepsilon (-\Delta)^m \varphi dx + \int_{B_2} u F_\epsilon(x) dx,
\end{equation}
where each term of $F_\varepsilon(x)$ involves the derivatives of $\eta_\varepsilon$ up to order $2m$, so it satisfies 
\begin{equation}\label{BF560}
|F_\varepsilon(x)| \leq C \varepsilon^{-2m} \cdot \chi_{\mathcal{N}_{3\varepsilon} \backslash \mathcal{N}_{\varepsilon}}(x). 
\end{equation}
Since $u^{\frac{n+2m}{n-2m}} \in L^1(B_2)$,   by the dominated convergence theorem
$$
\int_{B_2} u^{\frac{n+2m}{n-2m}} \psi_\varepsilon dx \to \int_{B_2} u^{\frac{n+2m}{n-2m}} \varphi dx ~~ \textmd{and} ~~  \int_{B_2} u\eta_\varepsilon (-\Delta)^m \varphi dx \to \int_{B_2} u  (-\Delta)^m \varphi dx
$$
as $\varepsilon \to 0$. On the other hand, by H\"{o}lder inequality  we have
$$
\aligned
\left| \int_{B_2} u F_\epsilon(x) dx  \right|  & \leq C \left( \int_{\mathcal{N}_{3\varepsilon} \backslash \mathcal{N}_\varepsilon } u^{\frac{n+2m}{n-2m}} dx \right)^{\frac{n-2m}{n+2m}}  \left( \int_{\mathcal{N}_{3\varepsilon} \backslash  \mathcal{N}_\varepsilon} |F_\varepsilon|^{\frac{n+2m}{4m}} dx \right)^{\frac{4m}{n+2m}} \\
& \leq C \varepsilon^{-2m} \cdot \mathcal{L}^n (\mathcal{N}_{3\varepsilon})^{\frac{4m}{n+2m}} \cdot \left( \int_{\mathcal{N}_{3\varepsilon} \backslash \mathcal{N}_\varepsilon } u^{\frac{n+2m}{n-2m}} dx \right)^{\frac{n-2m}{n+2m}} \\
&\leq C \varepsilon^{\frac{4m(n-\lambda)}{n+2m} -2m} \left( \int_{\mathcal{N}_{3\varepsilon} \backslash \mathcal{N}_\varepsilon } u^{\frac{n+2m}{n-2m}} dx \right)^{\frac{n-2m}{n+2m}} \to 0
\endaligned
$$
as $\varepsilon\to 0$ because of  $\frac{4m(n-\lambda)}{n+2m} -2m \geq 0$ and $\lim_{\varepsilon\to0} \mathcal{L}^n (\mathcal{N}_{3\varepsilon} \backslash \mathcal{N}_\varepsilon)=0$. Thus, $u$ is a distributional solution in the entire ball $B_2$. 
\end{proof}

Suppose $n >2m$. Let  $G_m(x, y)$ be the Green function of $(-\Delta)^m$ on $B_2$ under the Navier
boundary condition: 
\begin{equation}\label{green}
\begin{cases}
(-\Delta)^m G_m(x, \cdot) =\delta_x ~~ & \textmd{in }~ B_2, \\
G_m(x, \cdot)=-\Delta G_m(x, \cdot) = \cdots =(-\Delta)^{m-1} G_m(x, \cdot) =0 ~~ & \textmd{on}~ \partial B_2,
\end{cases}
\end{equation}
where $\delta_x$ is the Dirac measure to the point  $x \in B_2$. 
Then we have, for any $u\in C^{2m}(B_1) \cap C^{2m-2}(\overline{B}_1)$, 
$$
u(x)=\int_{B_2} G_m(x, y) (-\Delta)^m u(y) dy + \sum_{i=1}^m  \int_{\partial B_2} H_i(x, y) (-\Delta)^{i-1}u(y) dS_y, 
$$
where 
$$
H_i(x,y)= - \frac{\partial}{\partial\nu_y} (-\Delta_y)^{m-i} G_m(x, y)~~~~ \textmd{for} ~ x\in B_2, y\in \partial B_2. 
$$
By direct computations, we have
\begin{equation}\label{GrFu}
G_m(x,y)=c_{n,m} |x-y|^{2m-n} + A_m(x, y),
\end{equation}
 $c_{n,m}=\frac{\Gamma(\frac{n}{2} - m)}{2^{2m} \pi^{n/2} \Gamma(m)}$, $A_m(x, y)$ is smooth in $B_2 \times B_2$, and 
\begin{equation}\label{GrFu09}
H_i(x, y) \geq 0,~~~~~~ i=1, \dots, m. 
\end{equation}

\begin{proposition}\label{P202}
Assume as in Proposition \ref{P201}. Let $u\in C^{2m}(\overline{B}_2 \backslash \Lambda)$ be a positive solution of \eqref{HOE}.  Then
\begin{equation}\label{InRLo}
u(x)=\int_{B_2} G_m(x, y) u(y)^{\frac{n+2m}{n-2m}} dy + \sum_{i=1}^m  \int_{\partial B_2} H_i(x, y) (-\Delta)^{i-1}u(y) dS_y
\end{equation}
for all $x\in B_2 \backslash \Lambda$. 
\end{proposition}
\begin{proof}
For any $x\in B_2 \backslash \Lambda$, define 
$$
v(x) = \int_{B_2} G_m(x, y) u(y)^{\frac{n+2m}{n-2m}} dy + \sum_{i=1}^m  \int_{\partial B_2} H_i(x, y) (-\Delta)^{i-1}u(y) dS_y. 
$$
Since $u(y)^{\frac{n+2m}{n-2m}} \in L^1(B_2)$ and the Riesz potential $|x|^{2m-n}$ is weak type $\left( 1, \frac{n}{n-2m} \right)$, $v\in L^{\frac{n}{n-2m}}_{weak}(B_2) \cap L^1(B_2)$. Let $w=u-v$. Then, by Proposition \ref{P201},  $w$ satisfies 
$$
(-\Delta)^m w=0 ~~~~~ \textmd{in} ~ B_2
$$
in the distributional sense, i.e., for any $\varphi\in C_c^\infty(B_2)$,
$$
\int_{B_2} w (-\Delta)^m \varphi dx =0. 
$$
By the regularity for polyharmonic functions (see, e.g., Mitrea \cite{Dis-book}), we know that $w\in C^\infty(B_2)$ is smooth and satisfies  $(-\Delta)^m w=0$  pointwise  in  $B_2$. Since  $w=-\Delta w = \cdots =
(-\Delta)^{m-1} w =0$ on $\partial B_2$, $w\equiv 0$ and thus $u=v$ in $B_2 \backslash \Lambda$.  
\end{proof}

Further, if assume additionally  that  \eqref{HOE02} holds, then one can show that $u$ satisfies the integral equation \eqref{Int} in some local sense. Namely, 
\begin{theorem}\label{T-Lo094}
Assume as in Theorem \ref{Thm01}. Then there exists $\tau>0$ (independent of $x\in \Lambda$)  such  that for any $x_0 \in \Lambda$  we have
\begin{equation}\label{T_Lo672}
u(x) = c_{n,m} \int_{B_\tau(x_0)} \frac{u(y)^{\frac{n+2m}{n-2m}}}{|x - y|^{n-2m}} dy + h_1(x) ~~~~~\textmd{for} ~ x\in B_\tau(x_0) \backslash \Lambda, 
\end{equation}
where $h_1(x)$ is a positive smooth function in $B_\tau(x_0)$.  
\end{theorem}
 
\begin{proof}
Without loss of generality, we can assume that $u\in C^{2m}(\overline{B}_2 \backslash \Lambda)$ and $u > 0$ in $\overline{B}_2 \backslash \Lambda$. Otherwise, we just consider the equation in a smaller ball. 

It follows from the assumptions on the singular set $\Lambda$ that  $\mathcal{H}^{n-2} (\Lambda)=0$ and hence $\textmd{Cap}(\Lambda)=0$, where $\mathcal{H}^{n-2} $ is the $(n-2)$-dimensional Hausdorff measure on $\mathbb{R}^n$  and $\textmd{Cap}(\Lambda)$ is the Newton capacity of $\Lambda$ (see, e.g.,  \cite{EGbook}).      Since $u>0$ and $-\Delta u \geq 0$ in $B_1\backslash \Lambda$,  by the maximum principle (see, e.g., Lemma 2.1 of  \cite{Chen-Lin95}) 
$$
u(x) \geq c_1:=\inf_{\partial B_2} u >0~~~~~ \textmd{for} ~ \textmd{all} ~ x \in \overline{B}_2 \backslash \Lambda. 
$$
By Proposition \ref{P201}, $u^{\frac{n+2m}{n-2m}} \in L^1(B_2)$. Hence, there exits $0< \tau< \frac{1}{4}$ independent of $z\in B_1$ such that 
$$
\int_{B_\tau(z)} |A_m(x, y)| u(y)^{\frac{n+2m}{n-2m}} dy  < \frac{c_1}{2} ~~~~~ \textmd{for}~  \textmd{all}  ~ x \in B_\tau(z) \subset B_{3/2},
$$
where $A_m(x, y)$ is as in \eqref{GrFu}.  By Proposition \ref{P202}, for every  $x_0 \in \Lambda$  we can write 
$$
u(x)=c_{n,m} \int_{B_\tau(x_0)} \frac{u(y)^{\frac{n+2m}{n-2m}}}{|x - y|^{n-2m}} dy + h_1(x),~~~~~ \textmd{for}  ~ x\in B_\tau(x_0) \backslash \Lambda, 
$$
where
$$
\aligned
h_1(x)  &  =  \int_{B_\tau(x_0)} A_m(x, y) u(y)^{\frac{n+2m}{n-2m}}dy + c_{n,m} \int_{B_2 \backslash B_\tau(x_0)} G_m(x, y) u(y)^{\frac{n+2m}{n-2m}}dy \\
&  ~~~ + \sum_{i=1}^m  \int_{\partial B_2} H_i(x, y) (-\Delta)^{i-1}u(y) dS_y \\
&  \geq -\frac{c_1}{2} +  \int_{\partial B_2} H_1(x, y) u(y) dS_y \\
&  \geq  -\frac{c_1}{2} + \inf_{\partial B_2} u =\frac{c_1}{2} >0 ~~~~~~ \textmd{for} ~ x \in B_\tau(x_0). 
\endaligned
$$
It is easy to check that $h_1$ is a smooth function in $B_\tau(x_0)$ and satisfies $(-\Delta)^m h_1=0$ in $B_\tau(x_0)$.  This completes the proof. 
\end{proof}

\subsection {Global singular  solutions}\label{S2-02} 
In this subsection, we show that if $0\leq k \leq \frac{n-2m}{2}$,  then every global singular positive solution of  \eqref{Whole} satisfies the integral equation \eqref{Whole-Int}. 
\begin{proposition}\label{Glo-P01}
Suppose that $1 \leq m < n/2$ and $m$ is an integer. Let $0\leq k \leq \frac{n-2m}{2}$ and $u \in C^{2m} (\mathbb{R}^n \backslash \mathbb{R}^k)$ be a nonnegative  solution of \eqref{Whole}. Then $u\in L^{\frac{n+2m}{n-2m}}_{loc} (\mathbb{R}^n)$ and $u$ is a distributional solution in $\mathbb{R}^n$. 
\end{proposition} 

\begin{proof}
The proof is similar to that of Proposition \ref{P201}, the only difference is that we have to choose an additional truncation function.  Let $R>0$ and take $\varphi\in C_c^\infty(\mathbb{R}^n)$ such that $\varphi=1$ on $B_{R/2}$ and $\varphi=0$ on $B_{R}^c$. Denote $\mathcal{N}_r:=\{x\in \mathbb{R}^n ~|~\textmd{ dist}(x, \mathbb{R}^k) < r\}$.    For small $\varepsilon >0$,  as in the proof of Proposition \ref{P201},  consider  
\begin{equation}\label{rho}
\eta_\varepsilon(x):= 1 - \int_{\mathcal{N}_{2\varepsilon}} \rho_\varepsilon(x-y) dy, 
\end{equation}
where $\rho\in C_c^\infty(B_1)$ with $\int_{B_1} \rho dx=1$ and $\rho_\varepsilon(x)=\frac{1}{\varepsilon^n} \rho(\frac{x}{\varepsilon})$. Then $\eta_\varepsilon \in C^\infty(\mathbb{R}^n)$ is a non-negative function  and satisfies \eqref{rho01} and \eqref{rho02}. Let $\varphi_\varepsilon(x) = [(\varphi\eta_\varepsilon)(x)]^q$ with $q=\frac{n+2m}{2}$. Multiplying both sides of \eqref{Whole} by $\varphi_\varepsilon$ and using integration by parts, we obtain 
$$
\aligned
\int_{\mathbb{R}^n} u^{\frac{n+2m}{n-2m}} \varphi_\varepsilon dx & = \int_{\mathbb{R}^n} u (-\Delta)^m \varphi_\varepsilon dx \\
& \leq C \int_{B_R} u (\varphi_\varepsilon)^{\frac{n-2m}{n+2m}} dx + C \varepsilon^{-2m} \int_{B_R \cap (\mathcal{N}_{3\varepsilon} \backslash \mathcal{N}_\varepsilon)}  u (\varphi_\varepsilon)^{\frac{n-2m}{n+2m}} dx\\
& \leq C \left( 1 + \varepsilon^{\frac{4m(n-k)}{n+2m} - 2m} \right) \left(\int_{\mathbb{R}^n} u^{\frac{n+2m}{n-2m}} \varphi_\varepsilon dx \right)^{\frac{n-2m}{n+2m}} \\
& \leq C \left(\int_{\mathbb{R}^n} u^{\frac{n+2m}{n-2m}} \varphi_\varepsilon dx \right)^{\frac{n-2m}{n+2m}}, 
\endaligned
$$
from which it follows that
$$
\int_{B_{R/2} \cap \mathcal{N}_{3\varepsilon}^c } u^{\frac{n+2m}{n-2m}}  dx \leq \int_{\mathbb{R}^n} u^{\frac{n+2m}{n-2m}} \varphi_\varepsilon dx \leq C. 
$$
By sending $\varepsilon\to 0$, we obtain
$$
\int_{B_{R/2} } u^{\frac{n+2m}{n-2m}}  dx \leq C. 
$$
Thus, $u\in L^{\frac{n+2m}{n-2m}}_{loc} (\mathbb{R}^n)$ since $R>0$ is arbitrary.  The proof of which $u$ is a distributional solution in $\mathbb{R}^n$ is very similar to that of  Proposition \ref{P201}, so we omit the details. 
\end{proof}   

Now we give some growth estimates for solutions of \eqref{Whole} at infinity. 
\begin{lemma}\label{Glo-P02}
Assume as in Proposition \ref{Glo-P01}. Let $u\in C^{2m}(\mathbb{R}^n \backslash \mathbb{R}^k)$ be a non-negative solution of \eqref{Whole}. Then 
\begin{equation}\label{Glo-E001}
\int_{\mathbb{R}^n} \frac{u(x)}{1 + |x|^{\gamma}} dx < +\infty ~~~~~~ \textmd{for} ~ \textmd{every}  ~ \gamma > \frac{n+2m}{2}
\end{equation}
and 
\begin{equation}\label{Glo-E002}
\int_{\mathbb{R}^n} \frac{u(x)^{\frac{n+2m}{n-2m}}}{1 + |x|^{\gamma}} dx < +\infty ~~~~~~ \textmd{for} ~ \textmd{every}  ~ \gamma > \frac{n - 2m}{2}. 
\end{equation}
\end{lemma} 
\begin{proof}
The estimate \eqref{Glo-E001} follows from Lemma 5.5 in \cite{AGHW}. Next we show \eqref{Glo-E002}. By the proof of Lemma 5.5 in\cite{AGHW}, we have
\begin{equation}\label{RRR}
\int_{B_R} u(y)^{\frac{n+2m}{n-2m}} dy \leq C R^{n -\frac{n+2m}{2}}~~~~~~ \textmd{for} ~ \textmd{every} ~ R>0.
\end{equation}
Therefore, for every  $\gamma > \frac{n - 2m}{2}$,  using  \eqref{RRR}  we obtain 
$$
\aligned
\int_{\mathbb{R}^n} \frac{u(x)^{\frac{n+2m}{n-2m}}}{1 + |x|^{\gamma}} dx &= \int_{B_1} \frac{u(x)^{\frac{n+2m}{n-2m}}}{1 + |x|^{\gamma}} dx + \sum_{i=1}^\infty \int_{B_{2^i} \backslash B_{2^{i-1}}} \frac{u(x)^{\frac{n+2m}{n-2m}}}{1 + |x|^{\gamma}} dx \\
& \leq C  \int_{B_1} u(x)^{\frac{n+2m}{n-2m}} dx + \sum_{i=1}^\infty \int_{B_{2^i} \backslash B_{2^{i -1}}} u(x)^{\frac{n+2m}{n-2m}} dx  \cdot  2^{-\gamma(i-1)}\\
&\leq C + C \sum_{i=1}^\infty (2^i)^{\frac{n - 2m}{2} - \gamma}  < +\infty. 
\endaligned
$$
This completes the proof. 
\end{proof}

Let $0\leq k \leq \frac{n-2m}{2}$ and $u \in C^{2m} (\mathbb{R}^n \backslash \mathbb{R}^k)$ be a nonnegative  solution of \eqref{Whole}.  Then by  Lemma \ref{Glo-P02}  the following function
\begin{equation}\label{FS098}
v(x):=c_{n,m} \int_{\mathbb{R}^n} \frac{u(y)^{\frac{n+2m}{n-2m}}}{|x - y|^{n-2m}} dy
\end{equation}
is well-defined for every $x\in \mathbb{R}^n \backslash \mathbb{R}^k$, and it is continuous on $\mathbb{R}^n \backslash \mathbb{R}^k$.    In addition, for any $R>0$, we write $v$ as $v=v_{1,R} + v_{2,R}$,  where
$$
v_{1,R} (x) = \int_{B_{2R}} \frac{u(y)^{\frac{n+2m}{n-2m}}}{|x - y|^{n-2m}} dy ~~~~ \textmd{and}   ~~~~
v_{2,R} (x) = \int_{B_{2R}^c} \frac{u(y)^{\frac{n+2m}{n-2m}}}{|x - y|^{n-2m}} dy. 
$$
Since $u^{\frac{n+2m}{n-2m}} \in L^1(B_{2R})$,  we have  $v_{1,R} \in L^1(B_R)$.   From Lemma \ref{Glo-P02} we easily know $v_{2,R} \in L^\infty(B_R)$.  Hence we  obtain  $v\in L_{loc}^1(\mathbb{R}^n)$.  Define, for any $\gamma \in \mathbb{R}$, 
$$
L_\gamma(\mathbb{R}^n):= \left\{ u\in L^1_{loc}(\mathbb{R}^n) ~ \Big|  ~ \int_{\mathbb{R}^n} \frac{|u(x)|}{1+|x|^{n+2\gamma}} dx  < \infty \right\}. 
$$
Moreover, we   have the following property for $v$. 
\begin{lemma}\label{VVV}
Assume as in Proposition \ref{Glo-P01} and $v$ is defined by \eqref{FS098}.  Then we have $v\in L_0(\mathbb{R}^n)$. 
\end{lemma}
\begin{proof}
By  Fubini's theorem, we have
\begin{equation}\label{Fubi}
\int_{\mathbb{R}^n} \frac{v(x)}{1+|x|^{n}} dx =c_{n,m} \int_{\mathbb{R}^n} u(y)^{\frac{n+2m}{n-2m}} \left( \int_{\mathbb{R}^n} \frac{1}{|x-y|^{n-2m}} \frac{1}{1+|x|^n} dx\right) dy. 
\end{equation} 
If $|y| \leq 1$, then
$$
\aligned
\int_{\mathbb{R}^n} \frac{1}{|x-y|^{n-2m}} \frac{1}{1+|x|^n} dx & \leq \int_{B_3} \frac{1}{|x|^{n-2m}}dx + C \int_{B_2^c} \frac{1}{|x|^{n-2m}} \frac{1}{1+|x|^n} dx  \\ 
& \leq C <\infty.
\endaligned 
$$
If $|y| >1$, then
$$
\int_{\mathbb{R}^n} \frac{1}{|x-y|^{n-2m}} \frac{1}{1+|x|^n} dx  =: \sum_{i=1}^3 I_i,
$$
where
$$
I_1=\int_{ \{|x|\leq \frac{|y|}{2}\} }    \frac{1}{|x-y|^{n-2m}} \frac{1}{1+|x|^n}  dx \leq  \frac{C \ln (1 + |y|^n)}{|y|^{n-2m}},  
$$
$$
\aligned
I_2= \int_{\{ \frac{|y|}{2} < |x| < 2|y| \} }  \frac{1}{|x-y|^{n-2m}} \frac{1}{1+|x|^n} dx & \leq \frac{C}{|y|^{n}}\int_{\{ \frac{|y|}{2} < |x| < 2|y| \}}  \frac{1}{|x-y|^{n-2m}} dx \\ 
& \leq \frac{C}{|y|^{n-2m}},  
\endaligned
$$
and
$$
I_3=\int_{\{ |x|\geq  2|y| \}} \frac{1}{|x-y|^{n-2m}} \frac{1}{1+|x|^n} dx \leq C  \int_{\{ |x|\geq  2|y| \}} \frac{1}{|x|^{2n-2m}} dx \leq  \frac{C}{|y|^{n-2m}}. 
$$
All of these estimates together with \eqref{Fubi} and \eqref{Glo-E002} give
$$
\int_{\mathbb{R}^n} \frac{v(x)}{1+|x|^{n}} dx \leq C \int_{B_1} u(y)^{\frac{n+2m}{n-2m}} dy +C \int_{B_1^c}   \frac{\ln (1 + |y|^n) u(y)^{\frac{n+2m}{n-2m}} }{ |y|^{n-2m}} dy  <\infty. 
$$
Thus, we obtain $v \in L_0(\mathbb{R}^n)$. 
\end{proof}

Estimate \eqref{Glo-E002} and Lemma \ref{VVV} are an improvement of Lemma 5.4 and Lemma 5.6  in Ao-Gonz\'{a}lez-Hyder-Wei \cite{AGHW}, respectively, when the order  of the equation in \cite{AGHW}  is an integer.    Now we can show the integral representation for global singular  solutions of \eqref{Whole}. 
\begin{theorem}\label{T-Glo}
Suppose that $1 \leq m < n/2$ and $m$ is an integer. Let $0\leq k \leq \frac{n-2m}{2}$ and $u \in C^{2m} (\mathbb{R}^n \backslash \mathbb{R}^k)$ be a nonnegative  solution of \eqref{Whole}. Then $u\in L^{\frac{n+2m}{n-2m}}_{loc} (\mathbb{R}^n)$ and $u$  satisfies 
\begin{equation}\label{T-Glo098}
u(x) =c_{n,m} \int_{\mathbb{R}^n} \frac{u(y)^{\frac{n+2m}{n-2m}}}{|x-y|^{n-2m}} dy ~~~~~~\textmd{for} ~ x\in \mathbb{R}^n \backslash \mathbb{R}^k. 
\end{equation} 
\end{theorem} 
\begin{proof}
Our proof is inspired by that of Theorem 1.8 in \cite{AGHW}, where the superharmonicity property for the fractional Laplacian equations in $\mathbb{R}^n$ was showed.     Define $v$ as in \eqref{FS098}. Then $v\in L_{loc}^1(\mathbb{R}^n)$ and it  satisfies
\begin{equation}\label{Dis-V0}
(-\Delta)^m v =u^{\frac{n+2m}{n-2m}} ~~~~~~ \textmd{in} ~ \mathbb{R}^n
\end{equation}
in the distributional sense, i.e., for any $\varphi \in C_c^\infty(\mathbb{R}^n)$, 
$$
\int_{\mathbb{R}^n} v (-\Delta)^m \varphi dx= \int_{\mathbb{R}^n} u^{\frac{n+2m}{n-2m}} \varphi dx. 
$$
Let $w=u-v$. Then,  using Proposition \ref{Glo-P01} we have  that $(-\Delta)^m w =0$ in $\mathbb{R}^n$ in the distributional sense.  By the regularity of polyharmonic functions (see, e.g., \cite{Dis-book}), $w\in C^\infty(\mathbb{R}^n)$ and $(-\Delta)^m w \equiv 0$ in the classical sense.   On the other hand, from Lemmas \ref{Glo-P02} and \ref{VVV} we know  that both $u$ and $v$ belong to $L_0(\mathbb{R}^n)$ and hence $w \in L_0(\mathbb{R}^n)$.  It follows from a Liouville type theorem (see, e.g.,  \cite{AGHW,Dis-book}) that $w \equiv 0$ in $\mathbb{R}^n$.  Thus, we have proved that 
$$
u(x)=c_{n,m} \int_{\mathbb{R}^n} \frac{u(y)^{\frac{n+2m}{n-2m}}}{|x-y|^{n-2m}} dy,~~~~~ \textmd{for} ~ \mathcal{L}^n ~ a.e. ~ x \in \mathbb{R}^n. 
$$
Recall  that $u, v \in C(\mathbb{R}^n \backslash \mathbb{R}^k)$. Therefore \eqref{T-Glo098} holds for any  $x\in \mathbb{R}^n \backslash \mathbb{R}^k$.   
\end{proof}

\section{Local estimate near a singular set}\label{S3000} 
In this section, we first prove Theorem \ref{IEthm01} and then use it and the local integral representation in Subsection \ref{S2-01} to show Theorem \ref{Thm01}. 

For $x\in \mathbb{R}^n, \lambda >0$ and a function $u$,  we denote
$$
\xi^{x, \lambda}= x+ \frac{\lambda^2 (\xi - x)}{|\xi - x|^2} ~~~  \textmd{for} ~  \xi\neq x, ~~~~~~\Omega^{x,\lambda} = \{\xi^{x, \lambda}, \xi\in \Omega\},
$$
and 
$$
u_{x, \lambda}(\xi) = \left( \frac{\lambda}{|\xi - x|} \right)^{n-2\sigma} u(\xi^{x, \lambda}). 
$$
Note that $(\xi^{x,\lambda})^{x, \lambda} =\xi$ and $(u_{x,\lambda})_{x,\lambda} = u$. If $x=0$, we use the notation  $u_\lambda=u_{0,\lambda}$. 

Suppose $u \in L^{\frac{n+2\sigma}{n-2\sigma}}(B_2) \cap C(B_2 \backslash \Sigma)$  is a positive solution of \eqref{Int},   and suppose  $h\in C^1(B_2)$ is a positive function satisfying 
$$ 
|\nabla \ln h| \leq \tilde{C}  ~~~~~~ \textmd{in} ~ B_{3/2} 
$$
for some constant $\tilde{C} >0$.   If we extend $u$ to be identically $0$ outside $B_2$, then we have
\begin{equation}\label{A-01}
u(x) = \int_{\mathbb{R}^n} \frac{u(y)^{\frac{n+2\sigma}{n-2\sigma}}}{|x - y|^{n-2\sigma}} dy  + h(x)~~~~~~ \textmd{for}  ~ x\in B_2 \backslash \Sigma. 
\end{equation}
Recalling  the following two identities (see, e.g., \cite{Li04}), 
\begin{equation}\label{Id-01}
\left( \frac{\lambda}{|\xi - x|} \right)^{n-2\sigma} \int_{|z-x|\geq \lambda}  \frac{u(z)^{\frac{n+2\sigma}{n-2\sigma}}}{|\xi^{x,\lambda} - z|^{n-2\sigma}} dz = \int_{|z-x| \leq \lambda} \frac{u_{x,\lambda}(z)^{\frac{n+2\sigma}{n-2\sigma}}}{|\xi - z|^{n-2\sigma}} dz
\end{equation}
and
\begin{equation}\label{Id-02}
\left( \frac{\lambda}{|\xi - x|} \right)^{n-2\sigma} \int_{|z-x|\leq \lambda}  \frac{u(z)^{\frac{n+2\sigma}{n-2\sigma}}}{|\xi^{x,\lambda} - z|^{n-2\sigma}} dz = \int_{|z-x| \geq \lambda} \frac{u_{x,\lambda}(z)^{\frac{n+2\sigma}{n-2\sigma}}}{|\xi - z|^{n-2\sigma}} dz,
\end{equation}
one has
\begin{equation}\label{ABC-01}
u_{x,\lambda}(\xi) = \int_{\mathbb{R}^n} \frac{u_{x,\lambda}(z)^{\frac{n+2\sigma}{n-2\sigma}}}{|\xi - z|^{n-2\sigma}} dz  + h_{x,\lambda}(\xi)~~~~~~ \textmd{for}  ~  \xi \in (B_2 \backslash \Sigma)^{x, \lambda}. 
\end{equation}
Thus, for any $x\in B_1$ and $\lambda <1$, we have  for any  $\xi \in B_2 \backslash \left( \Sigma \cup \Sigma^{x,\lambda} \cup B_\lambda(x) \right)$ that 
$$
u(\xi) -u_{x,\lambda}(\xi) =\int_{|z-x|\geq \lambda} K(x, \lambda; \xi, z) \big[ u(z)^{\frac{n+2\sigma}{n-2\sigma}} - u_{x,\lambda}(z)^\frac{n+2\sigma}{n-2\sigma} \big]dz + h_{x,\lambda}(\xi) - h(\xi), 
$$
where
$$
K(x, \lambda; \xi, z)=\frac{1}{|\xi -z|^{n-2\sigma}} - \left(\frac{\lambda}{|\xi - x|}  \right)^{n-2\sigma} \frac{1}{|\xi^{x,\lambda} -z|^{n-2\sigma}}. 
$$
It is elementary to check that 
$$
K(x, \lambda; \xi, z) >0~~~~~ \textmd{for} ~ \textmd{all} ~ |\xi -x|, |z-x| >\lambda>0. 
$$

Next, we shall  prove Theorem \ref{IEthm01} using  the method of moving spheres introduced by Li and Zhu \cite{Li-Zhu,Li04} and some blow up arguments developed in Jin-Li-Xiong \cite{JLX17}. 

\vskip0.10in

\noindent{\it Proof of Theorem \ref{IEthm01}. }   Suppose by contradiction that there exists a sequence $\{x_j\}_{j=_1}^\infty \subset B_1 \backslash \Sigma$ such that 
$$
d_j:=\textmd{dist} (x_j, \Sigma) \to 0~~~~~ \textmd{as}  ~ j\to \infty, 
$$
but
$$
d_j^{\frac{n-2\sigma}{2}} u(x_j) \to \infty~~~~~ \textmd{as}  ~ j\to \infty. 
$$
We may assume, without  loss of generality,  that $0\in \Sigma$ and $x_j\to 0$ as $j\to \infty$. 

Consider 
$$
v_j(x) := \left( \frac{d_j}{2} - |x-x_j| \right)^{\frac{n-2\sigma}{2}} u(x),~~~~~~|x - x_j| \leq \frac{d_j}{2}. 
$$
Since $u$ is positive and continuous in $\overline{B}_{d_j /2}(x_j)$, we can find a point $\bar{x}_j \in \overline{B}_{d_j/2 }(x_j)$ satisfying  
$$
v_j(\bar{x}_j) = \max_{|x -x_j| \leq \frac{d_j}{2}} v_j(x) >0. 
$$
Let $2\mu_j:= \frac{d_j}{2} - |\bar{x}_j - x_j|$.   Then
$$
0 < 2\mu_j \leq \frac{d_j}{2}~~~~~~  \textmd{and} ~~~~~~ \frac{d_j}{2} - |x-x_j| \geq \mu_j~~~~\forall ~  |x-\bar{x}_j| \leq \mu_j. 
$$
By the definition of $v_j$, we have
\begin{equation}\label{SLo3-000} 
(2 \mu_j)^{\frac{n-2\sigma}{2}} u( \bar{x}_j ) = v_j(\bar{x}_j) \geq v_j(x) \geq \mu_j^{\frac{n-2\sigma}{2}}u(x)~~~~\forall ~ |x-\bar{x}_j| \leq \mu_j. 
\end{equation}
Hence, we have
\begin{equation}\label{SLo3-01} 
2^{\frac{n-2\sigma}{2}} u( \bar{x}_j ) \geq u(x)~~~~~~ \forall ~ |x-\bar{x}_j| \leq \mu_j. 
\end{equation}
We also have
\begin{equation}\label{SLo3-02}
(2 \mu_j)^{\frac{n-2\sigma}{2}} u( \bar{x}_j ) = v_j(\bar{x}_j) \geq v_j(x_j) = \left( \frac{d_j}{2} \right)^{\frac{n-2\sigma}{2}} u(x_j) \to \infty ~~~~\textmd{as}  ~ j\to \infty. 
\end{equation} 
Now, define
$$
w_j(y)=\frac{1}{u(\bar{x}_j)} u \left( \bar{x}_j + \frac{y}{u(\bar{x}_j)^{\frac{2}{n-2\sigma}}} \right), ~h_j(y)=\frac{1}{u(\bar{x}_j)} h \left( \bar{x}_j + \frac{y}{u(\bar{x}_j)^{\frac{2}{n-2\sigma}}} \right)~  \textmd{in}~ \Omega_j, 
$$
where 
$$
\Omega_j =\left\{y\in \mathbb{R}^n :   \bar{x}_j + \frac{y}{u(\bar{x}_j)^{\frac{2}{n-2\sigma}}} \in  B_2 \backslash \Sigma \right\}. 
$$
We extend $w_j$ to be $0$ outside of $\Omega_j $.  Then $w_j$ satisfies $w_j(0)=1$ and 
\begin{equation}\label{Swj01}
w_j(y) = \int_{\mathbb{R}^n} \frac{w_j(z)^{\frac{n+2\sigma}{n-2\sigma}}}{|y -z |^{n-2\sigma}} dz + h_j(y)~~~~~
\textmd{for} ~ y\in \Omega_j. 
\end{equation}
Moreover, it follows from \eqref{SLo3-01} and  \eqref{SLo3-02} that  
$$
\|h_j\|_{C^1(B_{R_j})} \to 0, ~~~~~ w_j(y) \leq 2^{\frac{n-2\sigma}{2}}~~~ \textmd{in} ~ B_{R_j}, 
$$
where 
$$
R_j := \mu_j u(\bar{x}_j)^{\frac{2}{n-2\sigma}} \to \infty ~~~ \textmd{as} ~ j \to \infty. 
$$

{\it Claim 1: } There exists a function $w>0$ such  that, after passing to a subsequence, $ w_j \to w$ in $C_{loc}^\alpha(\mathbb{R}^n)$
for some $\alpha >0$ and $w$ satisfies   
\begin{equation}\label{S3-W1}
w(y) = \int_{\mathbb{R}^n} \frac{w(z)^{\frac{n+2\sigma}{n-2\sigma}}}{|y -z |^{n-2\sigma}} dz~~~~~ \textmd{for} ~ y \in \mathbb{R}^n. 
\end{equation}

Since for any $R>0$ we have  $w_j(y)\leq 2^\frac{n-2\sigma}{2}$ in $B_R$ for all large $j$, by the regularity results in  Section 2.1  of \cite{JLX17}  there exists $w\geq 0$ such that,  after passing to a subsequence if necessary,   
$$
w_j \to w~~~~~ \textmd{in}  ~C_{loc}^{\alpha}(\mathbb{R}^n)
$$
for some $\alpha>0$.  Clearly   $w(0)=1$.    
To prove that $w$ satisfies the integral equation \eqref{S3-W1},  we follow the argument of Proposition 2.9 in  \cite{JLX17}.    Write \eqref{Swj01} as  
$$
w_j(y)=\int_{B_R} \frac{w_j(z)^{\frac{n+2\sigma}{n-2\sigma}}}{|y -z |^{n-2\sigma}} dz + h_j(R,y)~~~~~
\textmd{for} ~ y\in \Omega_j, 
$$
where
$$
h_j(R,y)=\int_{B_R^c} \frac{w_j(z)^{\frac{n+2\sigma}{n-2\sigma}}}{|y -z |^{n-2\sigma}} dz + h_j(y).
$$
Then,  for $y\in B_{R/2}$  we have 
$$
\aligned
h_j(R,y)&=\int_{B_R^c} \frac{|z|^{n-2\sigma}}{|y-z|^{n-2\sigma}}\frac{w_j(z)^{\frac{n+2\sigma}{n-2\sigma}}}{|z |^{n-2\sigma}} dz + h_j(y)\\
&\leq C \int_{B_R^c} \frac{w_j(z)^{\frac{n+2\sigma}{n-2\sigma}}}{|z |^{n-2\sigma}} dz + \|h_j\|_{L^\infty(B_{R/2})}\\
&\leq  C w_j(0) + \|h_j\|_{L^\infty(B_{R/2})}  
\endaligned
$$
for all large $j$.  Similarly,  for $y\in B_{R/2}$,
$$
|\nabla h_j(R,y)|\leq  C(R)  w_j(0) + \|\nabla h_j\|_{L^\infty(B_{R/2})}. 
$$
From these we get $\|h_j(R,\cdot)\|_{C^1(B_{R/2})}\leq C(R)$ for all $j$ large.  Thus,  after passing to a subsequence, $h_j(R,\cdot)\to h(R,\cdot)$ in $C^{1/2}(B_{R/2})$. Therefore, 
\begin{equation}\label{hRywyint}
h(R,y)=w(y)-\int_{B_R} \frac{w(z)^{\frac{n+2\sigma}{n-2\sigma}}}{|y -z |^{n-2\sigma}} dz
\end{equation}
for $y\in B_{R/2}$. Moreover, $h(R,y)$ is nonnegative and non-increasing in $R$. Notice that when $R>>|y|$,
$$
\aligned
\frac{R^{n-2\sigma}}{(R+|y|)^{n-2\sigma}}(h_j(R,0)-h_j(0))&\leq h_j(R,y)-h_j(y)\\
&\leq \frac{R^{n-2\sigma}}{(R-|y|)^{n-2\sigma}}(h_j(R,0)-h_j(0)).
\endaligned
$$
Let $j$ tend to $\infty$, we get
$$
\frac{R^{n-2\sigma}}{(R+|y|)^{n-2\sigma}}h(R,0)\leq h(R,y)\\
\leq \frac{R^{n-2\sigma}}{(R-|y|)^{n-2\sigma}}h(R,0), 
$$
which implies that $\lim_{R\to\infty}h(R,y)=\lim_{R\to\infty}h(R,0)=:c_0\geq  0$. Let $R$ tend to $\infty$ in \eqref{hRywyint}, from Lebesgue's  monotone convergence theorem,
$$
w(y)=\int_{\R^n} \frac{w(z)^{\frac{n+2\sigma}{n-2\sigma}}}{|y -z |^{n-2\sigma}} dz+c_0~~~~~ \textmd{for} ~ y \in \mathbb{R}^n. 
$$
We claim that $c_0=0$. If not, then $w(y)\geq c_0$ for all $y\in\R^n$ and thus 
$$
1=w(0)\geq \int_{\R^n} \frac{c_0^{\frac{n+2\sigma}{n-2\sigma}}}{|z |^{n-2\sigma}} dz =\infty.
$$
This is impossible.  Claim 1 is proved.

Since $w(0)=1$,  by the classification results in \cite{Li04} or \cite{CLO06}, we have
\begin{equation}\label{Class301}
w(y) = \left( \frac{1+ \mu^2 |y_0|^2 }{1+ \mu^2|y - y_0|^2 } \right)^{\frac{n-2\sigma}{2}}
\end{equation}
for some $\mu >0$ and some $y_0 \in \mathbb{R}^n$.

On the other hand,  we will  show that,   for every  $\lambda >0$, 
\begin{equation}\label{1-Local3}
w_\lambda(y) \leq w(y)~~~~~~ \forall ~  |y| \geq \lambda. 
\end{equation}
By the proof of Theorem 1.1 in \cite{Li04}, \eqref{1-Local3} implies that $w \equiv constant$. This contradicts to  \eqref{Class301}.

Let us fix $\lambda_0 >0$ arbitrarily. Then for all $j$ large, we have $0< \lambda_0 < \frac{R_j}{10}$. Denote
$$
\Xi_j:=  \left\{y\in \mathbb{R}^n :   \bar{x}_j + \frac{y}{u(\bar{x}_j)^{\frac{2}{n-2\sigma}}} \in  B_{1} \backslash \Sigma \right\} \subset \Omega_j. 
$$
We are going to show that for all sufficiently large $j$,
\begin{equation}\label{2-Local3}
(w_j)_{\lambda_0}(y) \leq w_j(y)~~~~~~ \forall ~ |y| \geq \lambda_0,  ~  y\in \Xi_j. 
\end{equation}
Then \eqref{1-Local3} follows from \eqref{2-Local3} by sending $j\to \infty$.  

It follows from the same arguments as in Lemma 3.1 of \cite{JX19}  that there exists a constant $\bar{r}>0$ depending only on $n$, $\sigma$, $\tilde{C}$ and $\lambda_0$  such that for all large $j$ with $u(\bar{x}_j)^{-\frac{2}{n-2\sigma}} < \bar{r}$ there holds
\begin{equation}\label{3-Local3}
(h_j)_{\lambda}(y) \leq h_j(y)~~~~~~ \forall ~  y\in \Xi_j \backslash B_\lambda,  ~ 0< \lambda \leq \lambda_0+\frac{1}{2}. 
\end{equation}

{\it Claim 2:}  There exists a real  number $\lambda_1>0$ independent of (large) $j$ such that for every $0< \lambda<\lambda_1$, we have
$$
(w_j)_{\lambda}(y) \leq w_j(y)~~~~~~ \textmd{in}~  \Xi_j \backslash B_\lambda. 
$$

Since $w_j \to w$ locally uniformly and $w$ is given in \eqref{Class301}, we have that $w_j \geq c_0>0$ on $B_1$ for all $j$ sufficiently large.  On the other hand, from the equation \eqref{Swj01} and the regularity results in \cite{JLX17}  we know that $|\nabla w_j| \leq C_0<\infty$ on $B_1$ for all $j$ sufficiently large.  By  the proof of Lemma 3.1 in \cite{JX19} (see ($20$) there), there exists a $r_0>0$ independent of (large) $j$ such that for all $0< \lambda \leq r_0$
\begin{equation}\label{4-Local3}
(w_j)_{\lambda}(y) \leq w_j(y), ~~~~~~  0< \lambda < |y| \leq r_0. 
\end{equation}
Since  $w_j \geq c_0 > 0$  on $B_1$ for all $j$ sufficiently large, we also have 
$$
w_j(y) \geq  c_0^{\frac{n+2\sigma}{n-2\sigma}}  \int_{B_1} |y - z|^{2\sigma-n} dz \geq \frac{1}{C} (1+|y|)^{2\sigma-n} ~~~~~~ \textmd{in}~ \Omega_j
$$
for some constant $C>0$.  Therefore, we can find a small $0< \lambda_1\leq r_0$ independent of (large) $j$ such that for every $0< \lambda< \lambda_1$
$$
(w_j)_{\lambda}(y) \leq \left( \frac{\lambda_1}{|y|} \right)^{n-2\sigma} \max_{B_{r_0}} w_j \leq C\left( \frac{\lambda_1}{|y|} \right)^{n-2\sigma} \leq w_j(y)   ~~~~ \textmd{for}~ \textmd{all} ~  y \in \Omega_j \backslash B_{r_0}.  
$$
Together with \eqref{4-Local3},   Claim 2 is proved. 

We define
$$
\bar{\lambda}_j:= \sup\{ 0< \mu \leq \lambda_0 ~ | ~ (w_j)_{\lambda}(y) \leq w_j(y),~~ \forall  ~ |y| \geq \lambda,~ y\in \Xi_j,~ \forall ~ 0< \lambda < \mu\}, 
$$
where $\lambda_0$ is fixed at the beginning. By Claim 2, $\bar{\lambda}_j$ is well defined and $\bar{\lambda}_j  \geq \lambda_1>0$ for all  sufficiently large $j$. 

\vskip0.1in

{\it Claim 3:} $\bar{\lambda}_j =\lambda_0$ for all  sufficiently large $j$. 
\vskip0.1in

By \eqref{Id-01} and \eqref{Id-02}, we have  for any $\bar{\lambda}_j \leq \lambda\leq \bar{\lambda}_j+\frac{1}{2}$ and $y\in \Xi_j \backslash B_\lambda$ that 
\begin{equation}\label{y-01}
\aligned
& w_j(y) - (w_j)_\lambda(y) \\
& ~ =\int_{B_\lambda^c} K(0,\lambda; y, z) \left( w_j(z)^{\frac{n+2\sigma}{n-2\sigma}} - (w_j)_\lambda(z)^{\frac{n+2\sigma}{n-2\sigma}} \right) dz + h_j(y) - (h_j)_\lambda(y) \\
&~ \geq \int_{\Xi_j \backslash B_\lambda} K(0,\lambda; y, z) \left( w_j(z)^{\frac{n+2\sigma}{n-2\sigma}} - (w_j)_\lambda(z)^{\frac{n+2\sigma}{n-2\sigma}} \right) dz + J(\lambda, w_j, y),
\endaligned
\end{equation}
where we have used \eqref{3-Local3} and  
\begin{equation}\label{y-02}
\aligned
J(\lambda, w_j, y)  & = \int_{ \mathbb{R}^n \backslash \Xi_j } K(0,\lambda; y, z) \left( w_j(z)^{\frac{n+2\sigma}{n-2\sigma}} - (w_j)_\lambda(z)^{\frac{n+2\sigma}{n-2\sigma}} \right) dz \\
& = \int_{ \Omega_j \backslash \Xi_j } K(0,\lambda; y, z) \left( w_j(z)^{\frac{n+2\sigma}{n-2\sigma}} - (w_j)_\lambda(z)^{\frac{n+2\sigma}{n-2\sigma}} \right) dz\\
& ~~~~  - \int_{ \Omega_j^c } K(0,\lambda; y, z) (w_j)_\lambda(z)^{\frac{n+2\sigma}{n-2\sigma}}  dz. 
\endaligned
\end{equation}
Let 
$$
\Sigma_j:=\left\{y\in \mathbb{R}^n :   \bar{x}_j + \frac{y}{u(\bar{x}_j)^{\frac{2}{n-2\sigma}}} \in  \Sigma \right\}. 
$$
Then $\mathcal{L}^n (\Sigma_j)=0$.   For any $z\in \mathbb{R}^n \backslash (\Xi_j \cup \Sigma_j)$  and $\bar{\lambda}_j \leq \lambda \leq \bar{\lambda}_j +1 $, we have $|z| \geq \frac{1}{2} u(\bar{x}_j)^{\frac{2}{n-2\sigma}}$ and thus
$$
(w_j)_\lambda(z) \leq \left( \frac{\lambda}{|z|} \right)^{n-2\sigma} \max_{B_{\lambda_0 +1}} w_j \leq C u(\bar{x}_j)^{-2}. 
$$
By the equation \eqref{Int}, we have
\begin{equation}\label{Esi-301}
u(x) \geq 4^{2\sigma-n} \int_{B_2} u(y)^{\frac{n+2\sigma}{n-2\sigma}} dy=:c_1> 0 ~~~~~ \textmd{for} ~ \textmd{all}  ~ x\in B_2 \backslash \Sigma,
\end{equation} 
and using the definition of $w_j$, we obtain
\begin{equation}\label{y-03}
w_j(y) \geq \frac{c_1}{u(\bar{x}_j)} ~~~~~~ \textmd{in} ~ \Omega_j \backslash \Xi_j. 
\end{equation}
Therefore, for large $j$, 
$$
w_j(z)^{\frac{n+2\sigma}{n-2\sigma}} - (w_j)_\lambda(z)^{\frac{n+2\sigma}{n-2\sigma}} \geq \frac{1}{2} w_j(z)^{\frac{n+2\sigma}{n-2\sigma}} ~~~~~~ \textmd{in} ~ \Omega_j \backslash \Xi_j. 
$$
Now, we claim that
\begin{equation}\label{y-05}
\aligned
J(\lambda, w_j, y)  & \geq \frac{1}{2}\left( \frac{c_1}{u(\bar{x}_j)} \right)^{\frac{n+2\sigma}{n-2\sigma}} \int_{\Omega_j \backslash \Xi_j} K(0,\lambda;y,z)dz - C \int_{\Omega_j^c} K(0,\lambda;y,z) \left(  \frac{\lambda}{|z|}\right)^{n+2\sigma} dz \\
& \geq 
\begin{cases}
C (|y| - \lambda) u(\bar{x}_j)^{-1}, ~~~~~   &\textmd{if}~  \lambda \leq |y| \leq \bar{\lambda}_j   +1, \\
C   u(\bar{x}_j)^{-1}, ~~~~~ & \textmd{if}~ |y| > \bar{\lambda}_{j} +1,  ~ y \in \Xi_j, 
\end{cases}
\endaligned
\end{equation}
where $C$ is a positive constant. 

Indeed, since $K(0,\lambda;y,z) =0$ for  $|y|=\lambda$ and 
$$
y\cdot \nabla_y K(0,\lambda;y,z) \Big|_{|y|=\lambda} = (n-2\sigma) |y-z|^{2\sigma-n-2}(|z|^2 - |y|^2) >0
$$
for $|z| \geq \bar{\lambda}_j +2$, and using the positivity and smoothness of $K$ we obtain
\begin{equation}\label{K011}
\frac{\delta_1}{|y -z|^{n-2\sigma}} (|y| - \lambda) \leq K(0,\lambda; y, z) \leq  \frac{\delta_2}{|y -z|^{n-2\sigma}} (|y| - \lambda)
\end{equation}
for $\bar{\lambda}_j \leq \lambda  \leq |y| \leq \bar{\lambda}_j+1$, $\bar{\lambda}_j+2 \leq |z| \leq  M <\infty$, where the positive constants $\delta_1$ and $\delta_2$ are independent of (large) $j$.  Moreover,   if $M$ is large, then 
$$
0< c_2 \leq y \cdot \nabla_y( |y-z|^{n-2\sigma} K(0,\lambda;y,z)) \leq C_2 < \infty
$$
for all  $|z|\geq M $,  $\bar{\lambda}_j \leq \lambda \leq |y| \leq \bar{\lambda}_j +1$. 
Hence,  \eqref{K011} also holds for $\bar{\lambda}_j \leq \lambda \leq |y| \leq \bar{\lambda}_j +1$, $|z|\geq M$.
On the other hand, by the definition of $K(0,\lambda;y,z)$, we can verify that for $|y|\geq \bar{\lambda}_j +1$ and $|z| \geq \bar{\lambda}_j +2$,
\begin{equation}\label{K022}
\frac{\delta_3}{|y -z|^{n-2\sigma}}  \leq K(0,\lambda; y, z) \leq  \frac{1}{|y -z|^{n-2\sigma}} 
\end{equation}
for some $\delta_3\in (0, 1)$ independent of (large) $j$.

Denote $\tau_j:=u(\bar{x}_j)^{\frac{2}{n-2\sigma}} $.  Then  for all  sufficiently  large $j$, $\lambda\leq |y| \leq \bar{\lambda}_j +1$ (recall that $\lambda \leq \bar{\lambda}_j  +\frac{1}{2}$)  we have
$$
\aligned
J(\lambda, w_j, y) & \geq \frac{1}{2}\left( \frac{c_1}{u(\bar{x}_j)} \right)^{\frac{n+2\sigma}{n-2\sigma}} \int_{\Omega_j \backslash \Xi_j}  \frac{\delta_1}{|y -z|^{n-2\sigma}} (|y| - \lambda)  dz \\
&~~~~~ - C \int_{\Omega_j^c} \frac{\delta_2}{|y -z|^{n-2\sigma}} (|y| - \lambda)  \left(  \frac{\lambda}{|z|}\right)^{n+2\sigma} dz \\
& \geq C (|y| - \lambda)  u(\bar{x}_j)^{-\frac{n+2\sigma}{n-2\sigma}}  \int_{\{\frac{5}{4} \tau_j \leq  |z| \leq \frac{7}{4}\tau_j \} \backslash \Sigma_j} \frac{1}{|y -z|^{n-2\sigma}}  dz \\
& ~~~~~ -C (|y| - \lambda) \int_{\{ |z| \geq \frac{7}{4}\tau_j \} \cup \Sigma_j}  \frac{1}{|y -z|^{n-2\sigma}} \left( \frac{1}{|z|}\right)^{n+2\sigma} dz\\
& \geq C (|y| - z) u(\bar{x}_j)^{-1} - C (|y| - z) u(\bar{x}_j)^{-\frac{2n}{n-2\sigma}} \\
& \geq C (|y| - z ) u(\bar{x}_j)^{-1}, 
\endaligned
$$
where we have used $\mathcal{L}^n(\Sigma_j) =0$ for all $j$ and $u(\bar{x}_j)  \to \infty$ as $j\to \infty$.   Similarly, for $|y| \geq \bar{\lambda}_j +1$ and $y \in \Xi_j$, we have
$$
J(\lambda, w_j, y)\geq C  u(\bar{x}_j)^{-1} - C  u(\bar{x}_j)^{-\frac{2n}{n-2\sigma}}  \geq C  u(\bar{x}_j)^{-1}. 
$$
Thus, \eqref{y-05} is verified. 

By \eqref{y-01} and \eqref{y-05},    there exists $\varepsilon_1(j)  \in (0, \frac{1}{2})$   such that 
$$
w_j(y) - (w_j)_{\bar{\lambda}_j}(y) \geq J(\bar{\lambda}_j, w_j,y) \geq C u(\bar{x}_j)^{-1} \geq \frac{\varepsilon_1(j)}{|y|^{n-2\sigma}} 
$$
for all $|y| \geq  \bar{\lambda}_j+1$, $y \in \Xi_j$.  This, together with the explicit formula of  $(w_j)_{\lambda}(y)$, yields  that there exists $0< \varepsilon_2(j) < \varepsilon_1(j)$ such that for any $\bar{\lambda}_j\leq \lambda \leq \bar{\lambda}_j+\varepsilon_2(j)$, 
\begin{equation}\label{1wj}
\aligned
w_j(y) - (w_j)_\lambda(y) & \geq \frac{\varepsilon_1(j)}{|y|^{n-2\sigma}} + \left(  (w_j)_{\bar{\lambda}_j}(y) - (w_j)_\lambda(y)  \right)\\
& \geq  \frac{\varepsilon_1(j)}{2|y|^{n-2\sigma}}~~~~~~ \forall ~ |y| \geq \bar{\lambda}_j +1,  ~ y\in \Xi_j. 
\endaligned
\end{equation}
For $\varepsilon_j \in \left( 0, \varepsilon_2(j) \right)$ which we choose below, by \eqref{y-01} and \eqref{y-05} we have, for $\bar{\lambda}_j\leq \lambda \leq \bar{\lambda}_j +\varepsilon_j $ and for $\lambda \leq |y| \leq \bar{\lambda}_j +1$, 
$$
\aligned
w_j(y) - (w_j)_\lambda(y) & \geq \int_{\lambda \leq |z| \leq \bar{\lambda}_j +1}  K(0,\lambda; y, z) \left( w_j(z)^{\frac{n+2\sigma}{n-2\sigma}} - (w_j)_\lambda(z)^{\frac{n+2\sigma}{n-2\sigma}} \right) dz\\
& ~~~ +     \int_{\bar{\lambda}_j +2  \leq |z| \leq \bar{\lambda}_j +3}  K(0,\lambda; y, z) \left( w_j(z)^{\frac{n+2\sigma}{n-2\sigma}} - (w_j)_\lambda(z)^{\frac{n+2\sigma}{n-2\sigma}} \right) dz \\
& \geq -C \int_{\lambda \leq |z| \leq \lambda +\varepsilon_j}   K(0,\lambda; y, z) (|z| -\lambda) dz\\
& ~~~ + \int_{\lambda+\varepsilon_j \leq |z| \leq \bar{\lambda}_j +1}  K(0,\lambda; y, z) \left( (w_j)_{\bar{\lambda}_j}(z)^{\frac{n+2\sigma}{n-2\sigma}} - (w_j)_\lambda(z)^{\frac{n+2\sigma}{n-2\sigma}} \right) dz \\
& ~~~ + \int_{ \bar{\lambda}_j +2  \leq |z| \leq \bar{\lambda}_j +3 }  K(0,\lambda; y, z) \left( w_j(z)^{\frac{n+2\sigma}{n-2\sigma}} - (w_j)_\lambda(z)^{\frac{n+2\sigma}{n-2\sigma}} \right) dz,
\endaligned
$$
where we have used 
$$
| w_j(z)^{\frac{n+2\sigma}{n-2\sigma}} - (w_j)_\lambda(z)^{\frac{n+2\sigma}{n-2\sigma}}  |\leq C(|z| - \lambda)
$$
in the second inequality. By \eqref{1wj}  there exists $\delta_j >0$ such that
$$
w_j(z)^{\frac{n+2\sigma}{n-2\sigma}} - (w_j)_\lambda(z)^{\frac{n+2\sigma}{n-2\sigma}} \geq \delta_j ~~~~~~ \textmd{for}  ~ \bar{\lambda}_j +2  \leq |z| \leq \bar{\lambda}_j +3. 
$$
Since $\|w_j\|_{C^1(B_{\lambda_0 + 2})}\leq C$ (independent of $j$),  there exists some constant $C>0$ independent of both $\varepsilon$ and $j$  such that for $\bar{\lambda}_j\leq \lambda \leq \bar{\lambda}_j +\varepsilon_j $, 
$$
| (w_j)_{\bar{\lambda}_j}(z)^{\frac{n+2\sigma}{n-2\sigma}} - (w_j)_\lambda(z)^{\frac{n+2\sigma}{n-2\sigma}}  |\leq C(\lambda - \bar{\lambda}_j)\leq C \varepsilon_j~~~~~ \forall ~ \lambda \leq |z| \leq \bar{\lambda}_j+1.
$$
For any  $\lambda \leq |y| \leq \bar{\lambda}_j+1$, one can estimate the integrals of the kernel $K$ (or, see \cite{JX19}):   
$$
\aligned
\int_{\lambda+\varepsilon_j  \leq |z| \leq \bar{\lambda}_j +1}  K(0,\lambda; y, z) dz &  \leq \left|  \int_{\lambda+\varepsilon_j \leq |z| \leq \bar{\lambda}_j +1}  \left( \frac{1}{|y -z|^{n-2\sigma}} -  \frac{1}{|y^{0,\lambda} -z|^{n-2\sigma}}\right) dz  \right| \\
& ~~~~ + \int_{\lambda+\varepsilon_j \leq |z| \leq \bar{\lambda}_j +1}  \left|  \left(\frac{\lambda}{|y|}\right)^{n-2\sigma} -1 \right|   \frac{1}{|y^{0,\lambda} -z|^{n-2\sigma}} dz\\
& \leq C (\varepsilon_j^{2\sigma-1}  +  |\ln \varepsilon_j|  +1) (|y| - \lambda) 
\endaligned
$$
and 
$$
\aligned
\int_{ \lambda \leq |z| \leq \lambda +\varepsilon_j }   K(0,\lambda; y, z) (|z| -\lambda) dz & \leq \left|  \int_{ \lambda \leq |z| \leq \lambda +\varepsilon_j }   \left( \frac{|z| -\lambda}{|y -z|^{n-2\sigma}} -  \frac{|z| -\lambda}{|y^{0,\lambda} -z|^{n-2\sigma}}\right) dz  \right| \\
& ~~~ + \varepsilon_j  \int_{ \lambda \leq |z| \leq \lambda +\varepsilon_j }  \left|  \left(\frac{\lambda}{|y|}\right)^{n-2\sigma} -1 \right|   \frac{1}{|y^{0,\lambda} -z|^{n-2\sigma}} dz\\
& \leq C (|y| - \lambda) \varepsilon_j^{2\sigma/n} + C\varepsilon_j(|y| - \lambda) \\
& \leq C (|y| - \lambda) \varepsilon_j^{2\sigma/n}. 
\endaligned
$$
Therefore,   using  \eqref{K011}  we have  for  $\lambda < |y| \leq \bar{\lambda}_j + 1$ that 
$$
\aligned
& w_j(y) - (w_j)_\lambda(y) \\
&~~ \geq -C \varepsilon_j^{2\sigma/n} (|y| - \lambda) + \delta_1\delta_j (|y| - \lambda) \int_{\bar{\lambda}_j +2  \leq |z| \leq \bar{\lambda}_j +3} \frac{1}{|y-z|} dz\\
& ~~ \geq \left(\delta_1\delta_jc -C \varepsilon_j^{2\sigma/n} \right) (|y| - \lambda) \geq 0
\endaligned
$$
if $\varepsilon_j$ is sufficiently small.  This  and  \eqref{1wj}  contradict to the definition of $\bar{\lambda}_j$ if $\bar{\lambda}_j < \lambda_0$ for sufficiently large $j$. Thus, we proved the Claim 3. 

It follows that  \eqref{1-Local3} holds  and the proof of Theorem \ref{IEthm01} is completed. 
\hfill$\square$

\vskip0.10in

Now we combine Theorem \ref{IEthm01}  with the local integral representation in Theorem \ref{T-Lo094}   to give the proof of Theorem \ref{Thm01}.    

\vskip0.10in

{\noindent} {\it Proof of Theorem \ref{Thm01}.} Let $u\in C^{2m}(B_2 \backslash \Lambda)$ be a positive solution of \eqref{HOE} satisfying \eqref{HOE02}.  By Theorem \ref{T-Lo094}  there exists $0< \tau < 1/4$ independent of $\Lambda$ such that  
for any $x_0 \in \Lambda$, we have (up to the constant $c_{n,m}$)  
\begin{equation}\label{SS-0368}
u(x) =  \int_{B_\tau(x_0)} \frac{u(y)^{\frac{n+2m}{n-2m}}}{|x - y|^{n-2m}} dy + h_1(x) ~~~~~\textmd{for} ~ x\in B_\tau(x_0) \backslash \Lambda, 
\end{equation}
where $h_1(x)$ is a positive smooth function in $B_\tau(x_0)$.  For any $x_0 \in \Lambda$,  define
$$
v(x) = \left( \frac{\tau}{2} \right)^{\frac{n-2m}{2}} u\left( \frac{\tau}{2} x + x_0 \right),~~ h(x) = \left( \frac{\tau}{2} \right)^{\frac{n-2m}{2}} h_1 \left( \frac{\tau}{2} x + x_0 \right)
~~~ \textmd{in} ~  B_2 \backslash \Sigma, 
$$
where 
$$
\Sigma:=\left\{x \in \mathbb{R}^n : \frac{\tau}{2} x +  x_0 \in \Lambda \right\}. 
$$
Note that, in general, $\Sigma$  intersects  the boundary $\partial B_2$. 
Then $v \in L^{\frac{n+2m}{n-2m}}(B_2) \cap C(B_2\backslash \Sigma) $ and $v$ satisfies \eqref{Int} with $\sigma=m$  in $B_2 \backslash \Sigma$,  and $h\in C^1(B_2)$ is a positive smooth function. Using Theorem \ref{IEthm01} for $v$, we know that there exists a constant $C>0$ such that 
$$
v(x) \leq C [ \textmd{dist}(x, \Sigma) ]^{-\frac{n-2m}{2}} ~~~~~\textmd{for} ~ \textmd{all} ~  x\in B_1 \backslash \Sigma. 
$$
Rescaling back to $u$, we have 
$$
u(x) \leq C [ \textmd{dist}(x, \Lambda) ]^{-\frac{n-2m}{2}} ~~~~~\textmd{for} ~ \textmd{all} ~  x\in B_\tau(x_0) \backslash \Lambda. 
$$
Since $\Lambda$ is a compact set, the desired estimate \eqref{Est01} follows by a finite covering argument. 
 \hfill$\square$

\section{Asymptotic symmetry for local singular solutions}\label{S4000}
In this section, we first show the asymptotic symmetry of local singular solutions for the integral equation \eqref{Int}  in Theorem \ref{IEthm02},  and then show the same thing for the differential equation \eqref{HOE} stated in Theorem \ref{Thm02} by combining Theorem \ref{IEthm02}  with  the integral representation in Theorem \ref{T-Lo094}.  We still use the notations introduced at  the beginning  of Section \ref{S3000}.   

\vskip0.1in

\noindent{\it Proof of Theorem \ref{IEthm02}. }  Without loss of generality, we may assume $0\in \Sigma$. We will show that there exists a small $ \varepsilon>0 $ such that for every $x \in \overline{B}_{1/4} \backslash \Sigma$, 
\begin{equation}\label{SA-01}
u_{x, \lambda} (y) \leq u(y) ~~~~~ \textmd{for}  ~ \textmd{all} ~ y\in B_{3/2} \backslash ( B_\lambda(x) \cup \Sigma), ~ 0< \lambda < \textmd{dist}(x, \Sigma) \leq  \varepsilon. 
\end{equation}

First of all, by Lemma 3.1 in \cite{JX19},  there exists a positive constant $0 < r_0 < \frac{1}{2}$ depending only on $n$, $\sigma$ and $\|\nabla\ln h\|_{L^\infty(B_{3/2})}$  such that for every $x\in B_1$ and $0< \lambda \leq r_0$ there holds
\begin{equation}\label{HH01}
h_{x,\lambda} (y) \leq h(y)~~~~~~~\forall ~ |y - x|\geq \lambda,~ y\in B_{3/2}. 
\end{equation}
Moreover, from the proof of Lemma 3.1 in \cite{JX19} (see ($20$) there) we know  that,  for every $x\in \overline{B}_{1/4} \backslash \Sigma$  there exists $0< r_x< \textmd{dist}(x, \Sigma)$ such that for all $0< \lambda\leq r_x$, 
\begin{equation}\label{3UU01}
u_{x,\lambda} (y) \leq u(y),  ~~~~~~~0< \lambda< |y - x| \leq  r_x. 
\end{equation}
By the equation \eqref{Int},  we have
\begin{equation}\label{3UU02}
u(x) \geq 4^{2\sigma-n} \int_{B_2} u(y)^{\frac{n+2\sigma}{n-2\sigma}} dy=:c_1> 0 ~~~~~ \textmd{for} ~ \textmd{all}  ~ x\in B_2 \backslash \Sigma,
\end{equation} 
and thus, we can find $0< \lambda_1 \ll r_x$ such that for all $0< \lambda \leq \lambda_1$, 
\begin{equation}\label{3UU03}
u_{x,\lambda} (y) \leq u(y),  ~~~~~~~ y\in B_{3/2} \backslash (B_{r_x}(x) \cup \Sigma). 
\end{equation}
Combining \eqref{3UU01} with  \eqref{3UU03},  we obtain that for every $0< \lambda\leq \lambda_1$, 
\begin{equation}\label{3UU04}
u_{x,\lambda} (y) \leq u(y),  ~~~~~~~ y\in B_{3/2} \backslash (B_\lambda(x) \cup \Sigma). 
\end{equation}
Therefore, 
$$
\aligned
\bar{\lambda}(x):=\sup\{ & 0< \mu < \textmd{dist}(x, \Sigma) ~ | ~ u_{x,\lambda} (y) \leq u(y), \forall ~ y\in B_{3/2} \backslash (B_\lambda(x) \cup \Sigma), \\
&~     \forall ~ 0< \lambda<\mu \}
\endaligned
$$
is well defined for any $x \in \overline{B}_{1/4} \backslash \Sigma$ and is positive. 

Next we show that there exists a small $\varepsilon >0$ such that $\bar{\lambda}(x)=\textmd{dist}(x, \Sigma)$ for all $x\in \overline{B}_{1/4}$  and  $0< \textmd{dist}(x, \Sigma) \leq \varepsilon$.   For brevity, we denote $\bar{\lambda}=\bar{\lambda}(x) $ in the below. 

For any  $x\in \overline{B}_{1/4}$ and $\bar{\lambda} \leq \lambda < \textmd{dist}(x, \Sigma) \leq r_0$,  by \eqref{HH01}, \eqref{Id-01} and \eqref{Id-02}  we have for $y \in B_{3/2}$ that 
$$
u(y) - u_{x,\lambda}(y) \geq \int_{B_1\backslash B_\lambda(x)}  K(x,\lambda;y,z) \left( u(z)^{\frac{n+2\sigma}{n-2\sigma}} - u_{x,\lambda}(z)^{\frac{n+2\sigma}{n-2\sigma}} \right) dz + J(\lambda, u, y), 
$$
where
$$
\aligned
J(\lambda, u, y)  & =\int_{B_2\backslash B_1} K(x,\lambda;y,z) \left( u(z)^{\frac{n+2\sigma}{n-2\sigma}} - u_{x,\lambda}(z)^{\frac{n+2\sigma}{n-2\sigma}}  \right) dz \\
&~~~~~   - \int_{B_2^c} K(x,\lambda;y,z)  u_{x,\lambda}(z)^{\frac{n+2\sigma}{n-2\sigma}} dz. 
\endaligned 
$$
For $y\in B_1^c$, $x\in \overline{B}_{1/4}$  and  $\bar{\lambda} \leq  \lambda < \textmd{dist}(x, \Sigma) < \frac{1}{10}$, we have 
$$
 \left| x + \frac{\lambda^2(y - x)}{ |y -x|^2 }  -x \right| \leq \frac{4}{3} \lambda^2 < \frac{1}{2} \textmd{dist}(x, \Sigma) 
$$
and 
$$
 \left| x + \frac{\lambda^2(y - x)}{ |y -x|^2 } \right| \leq \frac{1}{2} \textmd{dist}(x, \Sigma) + |x| <1. 
$$
Hence 
$$
\textmd{dist}\left( x + \frac{\lambda^2(y - x)}{ |y -x|^2 }, \Sigma \right) \leq \left| x + \frac{\lambda^2(y - x)}{ |y -x|^2 }  -x \right| + \textmd{dist}(x, \Sigma)  \leq \frac{3}{2} \textmd{dist}(x, \Sigma)
$$
and
$$
\textmd{dist}\left( x + \frac{\lambda^2(y - x)}{ |y -x|^2 }, \Sigma \right) \geq  \textmd{dist}(x, \Sigma) - \left| x + \frac{\lambda^2(y - x)}{ |y -x|^2 }  -x \right|  \geq \frac{1}{2} \textmd{dist}(x, \Sigma).  
$$
It follows from Theorem \ref{IEthm01} that 
$$
u \left( x + \frac{\lambda^2(y - x)}{ |y -x|^2 } \right) \leq C \textmd{dist}(x, \Sigma)^{-\frac{n-2\sigma}{2}}. 
$$
Thus, for all $y \in B_1^c$, 
\begin{equation}\label{3UU090}
\aligned
u_{x,\lambda} (y) &= \left( \frac{\lambda}{|y - x|} \right)^{n-2\sigma} u \left( x + \frac{\lambda^2(y - x)}{ |y -x|^2 } \right) \\
& \leq  C  \lambda^{n-2\sigma} \textmd{dist}(x, \Sigma)^{-\frac{n-2\sigma}{2}}  \leq C \textmd{dist}(x, \Sigma)^{\frac{n-2\sigma}{2}} \leq C \varepsilon^{\frac{n-2\sigma}{2}}
\endaligned 
\end{equation}
for any  $x\in \overline{B}_{1/4}$ and $\bar{\lambda} \leq  \lambda < \textmd{dist}(x, \Sigma) < \varepsilon< \frac{1}{10}$.  By \eqref{3UU02} we obtain that for any $x\in \overline{B}_{1/4}$, 
\begin{equation}\label{MS001}
u_{x,\lambda} (y) \leq C \varepsilon^{\frac{n-2\sigma}{2}} \leq \frac{c_1}{2} <  u(y) ~~~~ \forall~ y \in B_2\backslash (B_1 \cup \Sigma),   ~ \forall ~ \bar{\lambda}\leq  \lambda< \textmd{dist}(x, \Sigma) \leq \varepsilon 
\end{equation}
if $\varepsilon$ is sufficiently small.

For $y\in B_1\backslash (B_\lambda(x) \cup \Sigma)$, $x\in \overline{B}_{1/4}$ and $\bar{\lambda} \leq  \lambda < \textmd{dist}(x, \Sigma) < \varepsilon$, by \eqref{3UU02} and \eqref{3UU090}, using the similar arguments as in proving \eqref{y-05} and noticing that $\mathcal{L}^n(\Sigma)=0$,   we have
\begin{equation}\label{MS002}
\aligned
J(\lambda, u, y) & \geq \int_{B_2\backslash B_1} K(x,\lambda;y,z) \left( c_1^{\frac{n+2\sigma}{n-2\sigma}} - C \varepsilon^{\frac{n+2\sigma}{2}}  \right) dz \\
&~~~~~   - C \int_{B_2^c} K(x,\lambda;y,z)  \left(\frac{1}{|z -x|} \right)^{n+2\sigma} [\textmd{dist}(x, \Sigma)]^{\frac{n+2\sigma}{2}} dz \\
& \geq \frac{1}{2} c_1^{\frac{n+2\sigma}{n-2\sigma}}  \int_{B_2\backslash B_1} K(x,\lambda;y,z) dz - C\varepsilon^{\frac{n+2\sigma}{2}} \int_{B_2^c} K(x,\lambda;y,z)  \frac{1}{|z -x|^{n+2\sigma}} dz \\
& \geq  \frac{1}{2} c_1^{\frac{n+2\sigma}{n-2\sigma}}  \int_{B_{7/4}\backslash B_{5/4}} K(0,\lambda;y-x,z) dz\\
&~~~~~ - C\varepsilon^{\frac{n+2\sigma}{2}} \int_{B_{7/4}^c} K(0,\lambda;y-x,z)  \frac{1}{|z|^{n+2\sigma}} dz\\
& \geq C_2 (|y -x| - \lambda),
\endaligned
\end{equation}
if we let $\varepsilon$ be sufficiently small, where $C_2$ is a positive constant independent of $x$. If $\bar{\lambda} < \textmd{dist}(x, \Sigma) \leq  \varepsilon$ for some $x\in \overline{B}_{1/4}$, using \eqref{MS001} and  \eqref{MS002} with  the integral estimates techniques as in the proof of Theorem \ref{IEthm01}, the moving sphere procedure may continue beyond $\bar{\lambda}$ where we get a contradiction. Thus, we obtain $\bar{\lambda}(x)=\textmd{dist}(x, \Sigma)$ for $x\in \overline{B}_{1/4}$ and $0< \textmd{dist}(x, \Sigma) \leq \varepsilon$, where $\varepsilon$ is sufficiently small.  Therefore, \eqref{SA-01} is proved.  

Let $r>0$ small (less that $\varepsilon^2$), $x_1, x_2\in \Pi_r^{-1} (z)$ with $z \in \overline{B}_{1/8} \cap \Sigma$ be such that
$$
u(x_1)=\max_{\Pi_r^{-1} (z)} u(x),~~~~~ u(x_2)=\min_{\Pi_r^{-1} (z)} u(x). 
$$
Let $e_1=x_1 -z$, $e_2=x_2-z$, $x_3=x_1 + \frac{\varepsilon(e_1 - e_2)}{4|e_1- _2|}$. Then $e_1, e_2 \in (T_z\Sigma)^\bot$ and thus, $e_2 -e_1\in (T_z\Sigma)^\bot$. Let $\lambda=\sqrt{\frac{\varepsilon}{4} (|e_1 -e_2| + \frac{\varepsilon}{4})}$. It is easy to check that  $0< \lambda < |x_3 -z|= \textmd{dist}(x_3, \Sigma)< \varepsilon$ and $|x_3| < 1/4$.   From \eqref{SA-01} we obtain 
$$
u_{x_3,\lambda}(x_2) \leq u(x_2). 
$$
On the other hand,  the definition of $u_{x_3,\lambda}$ gives 
$$
\aligned
u_{x_3,\lambda}(x_2) & = \left( \frac{\lambda}{|e_1 - e_2| + \varepsilon/4} \right)^{n-2\sigma} u(x_1)\\
& = \left( \frac{1}{ 4|e_1 - e_2|/\varepsilon + 1 } \right)^{\frac{n-2\sigma}{2}} u(x_1)\\
& \geq \left( \frac{1}{ 8r/\varepsilon + 1 } \right)^{\frac{n-2\sigma}{2}} u(x_1).
\endaligned
$$
Thus, 
$$
\max_{\Pi_r^{-1} (z)} u(x)  \leq (8r/\varepsilon + 1 )^{\frac{n-2\sigma}{2}}  \min_{\Pi_r^{-1} (z)} u(x). 
$$
This implies that
$$
u(x)=u(x^\prime) (1 + O(r)) ~~~~~ \textmd{for} ~ \textmd{all} ~ x, x^\prime \in \Pi_r^{-1} (z)~~~~ \textmd{as}  ~r\to 0,
$$
where $O(r)$ is uniform for $z\in \overline{B}_{1/8} \cap \Sigma$. The proof of Theorem \ref{IEthm02} is completed. 
\hfill$\square$

\vskip0.1in

\noindent{\it Proof of Theorem \ref{Thm02}.}  The proof is very similar to that of Theorem \ref{Thm01}. Using Theorems \ref{IEthm02}  and \ref{T-Lo094},   by a rescaling argument and a  covering  argument,  there exists a small real number $\varepsilon >0$ such that 
$$
\max_{\Pi_r^{-1} (z)} u(x)  \leq (8r/\varepsilon + 1 )^{\frac{n-2\sigma}{2}}  \min_{\Pi_r^{-1} (z)} u(x)
$$
for all $z\in \Lambda$ and small $r>0$. Thus, we have
$$
u(x)=u(x^\prime) (1 + O(r)) ~~~~~ \textmd{for} ~ \textmd{all} ~ x, x^\prime \in \Pi_r^{-1} (z)~~~~ \textmd{as}  ~r\to 0,
$$
where $O(r)$ is uniform for $z\in \Lambda$. Theorem \ref{Thm02} is proved. 
\hfill$\square$

\section{Symmetry for global singular solutions}\label{S5000}
In this section, we prove Theorem \ref{IEthm03} and then  Theorem \ref{Thm03} follows immediately by using Theorem \ref{T-Glo}.   Finally,  we also give the proof of Corollary \ref{Cor01}.  

\vskip0.1in

\noindent{\it Proof of Theorem \ref{IEthm03}}. Without loss of generality, we assume that
\begin{equation}\label{55-S01}
\mathop{\limsup_{x \in \mathbb{R}^n \backslash \mathbb{R}^k}}\limits_{ x \to 0}  u(x) =\infty,
\end{equation}
where  $\mathbb{R}^k$ is a $k$-dimensional subspace of $\mathbb{R}^n$ with $0 \leq k \leq n-1$.  Denote $\mathbb{R}^{n-k} = (\mathbb{R}^k)^\bot$ as the orthogonal complement  of $\mathbb{R}^k$.  
\vskip0.10in

{\it Claim 1.} For every $x \in \mathbb{R}^{n-k} \backslash \{0\}$, there exists a real number $\lambda_2 \in (0, |x|)$ such that for any $0< \lambda < \lambda_2$, we have
\begin{equation}\label{5SC101}
u_{x,\lambda} (y) \leq u(y) ~~~~~ \forall ~ |y-x| \geq \lambda,~ y \in \mathbb{R}^n \backslash \mathbb{R}^k. 
\end{equation}
The proof of Claim 1 consists of two steps. 

{\it Step 1.}  We show that there exists $0< \lambda_1 < |x|$ such that for any $0< \lambda<\lambda_1$, 
\begin{equation}\label{5S0-0-1}
u_{x,\lambda} (y) \leq u(y)~~~~~ \forall ~0< \lambda\leq  |y - x| \leq \lambda_1.
\end{equation}
By Theorem 2.5 in \cite{JLX17}, we know that $u \in C^1({\mathbb{R}^n \backslash \mathbb{R}^k})$.  Suppose
$$
|\nabla \ln u| \leq C_1 ~~~~~ \textmd{in} ~ B_{|x|/2}(x)
$$  
for some constant $C_1>0$.  Then  we have
\begin{equation}\label{Mono01}
\aligned
\frac{d}{dr} (r^{\frac{n-2\sigma}{2}} u(x + r\theta)) &= r^{\frac{n-2\sigma}{2} -1} u(x + r\theta) \left( \frac{n-2\sigma}{2} - r \frac{\nabla u \cdot \theta}{u}\right) \\
&\geq r^{\frac{n-2\sigma}{2} -1} u(x + r\theta) \left( \frac{n-2\sigma}{2} - C_1 r \right) >0
\endaligned
\end{equation}
for all $0 < r < \lambda_1 :=\min\{ \frac{n-2\sigma}{2C_1}, \frac{|x|}{2}\}$ and $\theta \in \mathbb{S}^{n-1}$. For any $y \in B_{\lambda_1}(x)$, $0< \lambda < |y - x| \leq \lambda_1$, let $\theta=\frac{y-x}{|y-x|}$, $r_1 = |y - x|$ and $r_2=\frac{\lambda^2 r_1}{|y - x|^2}$.  Using \eqref{Mono01}  we have
$$
r_2^{\frac{n-2\sigma}{2}} u(x + r_2\theta) < r_1^{\frac{n-2\sigma}{2}} u(x + r_1 \theta). 
$$
That is,
$$
u_{x,\lambda} (y) \leq u(y), ~~~~~~0< \lambda\leq  |y - x| \leq \lambda_1.
$$

{\it Step 2.} We show that  there exists  $0< \lambda_2 <  \lambda_1  <|x| $ such that \eqref{5SC101} holds for all $0< \lambda < \lambda_2$.  By Fatou lemma, 
$$
 \mathop{ \liminf_{ x \in \mathbb{R}^n \backslash \mathbb{R}^k} }\limits_{ {|x| \to \infty} }   |x|^{n-2\sigma} u(x) =  \mathop{ \liminf_{ x \in \mathbb{R}^n \backslash \mathbb{R}^k} }\limits_{ {|x| \to \infty} }  \int_{\mathbb{R}^n} \frac{|x|^{n-2\sigma} u(y)^{\frac{n-2\sigma}{n-2\sigma}} }{|x - y|^{n-2\sigma}} dy  \geq \int_{\mathbb{R}^n} u(y)^{\frac{n-2\sigma}{n-2\sigma}} dy >0. 
$$
Consequently,   there exist two constants $c_1, R_1>0$ such that
\begin{equation}\label{5SC102}
u(y) \geq \frac{c_1}{|y|^{n-2\sigma}}~~~~~~~ \textmd{for} ~  \textmd{all} ~  |y|\geq R_1~  \textmd{and}  ~ y \in \mathbb{R}^n \backslash \mathbb{R}^k. 
\end{equation}
On the other hand, if $y \in B_{R_1} \backslash \mathbb{R}^k$, then by the equation \eqref{Whole-Int} and the positivity of $u$,  we obtain
$$
u(y) \geq \int_{B_{R_1}} \frac{u(z)^{\frac{n-2\sigma}{n-2\sigma}} }{|y - z|^{n-2\sigma}} dz \geq (2R_1)^{2\sigma -n} \int_{B_{R_1}} u(z)^{\frac{n-2\sigma}{n-2\sigma}} dz  >0.
$$
This, together with \eqref{5SC102}, implies that there exists $C>0$ such that 
$$
u(y) \geq \frac{C}{|y - x|^{n-2\sigma}} ~~~~~ \forall ~ |y - x| \geq \lambda_1, ~ y \in \mathbb{R}^n \backslash \mathbb{R}^k.
$$
Thus,  for sufficiently small  $\lambda_2 \in (0, \lambda_1 )$ and for any $0< \lambda< \lambda_2$,
$$
\aligned
u_{x,\lambda}(y) & = \left( \frac{\lambda}{|y - x|} \right)^{n-2\sigma} u \left( x + \frac{\lambda^2(y - x)}{ |y- x|^2 } \right) \\
& \leq \left( \frac{\lambda_2}{|y - x|} \right)^{n-2\sigma} \sup_{B_{\lambda_1}(x)} \leq u(y), ~~~ \forall ~ |y-x| \geq \lambda_1, ~ y \in \mathbb{R}^n \backslash \mathbb{R}^k. 
\endaligned 
$$
Estimate \eqref{5SC101} follows from \eqref{5S0-0-1} and the above.  Claim 1 is proved.

Now,  we can define 
$$
\aligned 
\bar{\lambda}(x):= \sup\{ & 0< \mu\leq |x|~ | ~ u_{x,\lambda}(y) \leq u(y),~ \forall ~ |y - x| \geq \lambda,~ y \in \mathbb{R}^n \backslash \mathbb{R}^k, \\
&~ \forall ~ 0< \lambda<\mu\}. 
\endaligned
$$
By Claim 1, $\bar{\lambda}(x)$ is well defined and $\bar{\lambda}(x)>0$. 

\vskip0.10in

{\it Claim 2.} $\bar{\lambda}(x) =|x|$ for all $x \in \mathbb{R}^{n-k} \backslash \{0\}$. 

\vskip0.10in

Suppose $\bar{\lambda}(x) < |x|$ for some $x \in \mathbb{R}^{n-k} \backslash \{0\}$. For brevity, we will denote $\bar{\lambda}=\bar{\lambda}(x) $ in the below.  By the definition of $\bar{\lambda}$, 
\begin{equation}\label{5CL6-01}
u_{x, \bar{\lambda} } (y) \leq u(y) ~~~~~~ \textmd{for} ~ \textmd{all} ~ |y - x|\geq \bar{\lambda}, ~ y \in \mathbb{R}^n \backslash \mathbb{R}^k.
\end{equation} 
Because of \eqref{55-S01}, we know that $u_{x, \bar{\lambda}} (y) \not\equiv u(y)$. For any $\bar{\lambda} \leq \lambda < |x|$, $y\in \mathbb{R}^n \backslash \mathbb{R}^k$ with $|y -x|\geq \lambda$,  by \eqref{Id-01} and \eqref{Id-02},  we have 
$$
u(y) - u_{x,\lambda}(y) 
=\int_{|z -x|\geq \lambda}  K(x,\lambda; y, z) \left( u(z)^{\frac{n+2\sigma}{n-2\sigma}} - u_{x,\lambda}(z)^{\frac{n+2\sigma}{n-2\sigma}} \right) dz. 
$$
It follows from the positivity of the kernel $K$ and $\mathcal{L}^n (\mathbb{R}^k)=0$  that 
$$
u_{x, \bar{\lambda} } (y) <  u(y) ~~~~~~ \textmd{for} ~ \textmd{all} ~ |y - x|\geq \bar{\lambda}, ~ y \in \mathbb{R}^n \backslash \mathbb{R}^k.
$$
Using Fatou lemma again, 
$$
\aligned
 &\mathop{ \liminf_{ y \in \mathbb{R}^n \backslash \mathbb{R}^k} }\limits_{ {|y| \to \infty} }   |y-x|^{n-2\sigma} ( u -u_{x,\bar{\lambda}} )(y) \\
 & ~~ =  \mathop{ \liminf_{ y \in \mathbb{R}^n \backslash \mathbb{R}^k} }\limits_{ {|y| \to \infty} }  
 \int_{|z -x|\geq \bar{\lambda}} |y-x|^{n-2\sigma}  K(x, \bar{\lambda}; y, z)  \left( u(z)^{\frac{n+2\sigma}{n-2\sigma}} - u_{x, \bar{\lambda}}(z)^{\frac{n+2\sigma}{n-2\sigma}} \right) dz \\
 &~~ \geq \int_{|z -x|\geq \bar{\lambda}} \left[ 1 - \left( \frac{\bar{\lambda}}{|z - x |} \right)^{n - 2\sigma} \right] \left( u(z)^{\frac{n+2\sigma}{n-2\sigma}} - u_{x, \bar{\lambda}}(z)^{\frac{n+2\sigma}{n-2\sigma}} \right) dz >0. 
\endaligned
$$
Consequently, there exist two constants $c_2, R_2 >0$ such that
\begin{equation}\label{5CL6-02}
( u -u_{x,\bar{\lambda}} )(y)  \geq \frac{c_2}{|y -x|^{n-2\sigma}}~~~~~~  \textmd{for} ~ \textmd{all} ~  |y -x |\geq R_2, ~ y \in \mathbb{R}^n \backslash \mathbb{R}^k, 
\end{equation} 
and for $\bar{\lambda} + 1\leq |y -x| \leq R_2$,  $y \in \mathbb{R}^n \backslash \mathbb{R}^k$, by \eqref{K022}  we have
$$
\aligned
u(y) - u_{x,\bar{\lambda}}(y)  &= \int_{|z -x|\geq \bar{\lambda} }  K(x, \bar{\lambda}; y, z) \left( u(z)^{\frac{n+2\sigma}{n-2\sigma}} - u_{x, \bar{\lambda}}(z)^{\frac{n+2\sigma}{n-2\sigma}} \right) dz \\
& \geq  \int_{ \bar{\lambda} +2 \leq |z -x| \leq \bar{\lambda} +8 }  \frac{\delta_3}{|y - z|^{n-2\sigma}} \left( u(z)^{\frac{n+2\sigma}{n-2\sigma}} - u_{x, \bar{\lambda}}(z)^{\frac{n+2\sigma}{n-2\sigma}} \right) dz \\
& \geq C_2 \int_{ \bar{\lambda} +2 \leq |z -x| \leq \bar{\lambda} +8 }  \left( u(z)^{\frac{n+2\sigma}{n-2\sigma}} - u_{x, \bar{\lambda}}(z)^{\frac{n+2\sigma}{n-2\sigma}} \right) dz >0
\endaligned
$$
for some $C_2>0$. Combining this with \eqref{5CL6-02}, we obtain that  there exists $\varepsilon_1 \in (0, 1)$ such that
\begin{equation}\label{5CL6-03}
(u - u_{x, \bar{\lambda}}) (y) \geq \frac{\varepsilon_1}{|y -x|^{n-2\sigma}} ~~~~~~  \textmd{for} ~ \textmd{all} ~  |y -x |\geq \bar{\lambda} +1, ~ y \in \mathbb{R}^n \backslash \mathbb{R}^k. 
\end{equation}
By \eqref{5CL6-03} and the explicit formula of $u_{x, \lambda}$,  there exists $0 < \varepsilon_2 < \varepsilon_1$ such that for all $\bar{\lambda} \leq \lambda \leq \bar{\lambda} + \varepsilon_2 < |x|$,  
\begin{equation}\label{5CL6-04}
\aligned
(u - u_{x, \lambda}) (y) & \geq \frac{\varepsilon_1}{|y -x|^{n-2\sigma}} + ( u_{x, \bar{\lambda}} - u_{x, \lambda}) (y) \\
 & \geq  \frac{\varepsilon_1}{2|y -x|^{n-2\sigma}}  ~~~~~~ \forall ~  |y -x |\geq \bar{\lambda} +1, ~ y \in \mathbb{R}^n \backslash \mathbb{R}^k. 
\endaligned
\end{equation} 
For $\varepsilon \in (0, \varepsilon_2)$ which we choose below, we have, for $\bar{\lambda} \leq \lambda \leq \bar{\lambda} + \varepsilon$ and for $y \in  \mathbb{R}^n \backslash \mathbb{R}^k$ with $\lambda \leq |y -x| \leq \bar{\lambda} +1$, 
$$
\aligned
u(y) - u_{x,\lambda}(y) 
& =\int_{|z -x|\geq \lambda}  K(x,\lambda; y, z) \left( u(z)^{\frac{n+2\sigma}{n-2\sigma}} - u_{x,\lambda}(z)^{\frac{n+2\sigma}{n-2\sigma}} \right) dz\\
&  \geq  \int_{\lambda \leq |z -x| \leq \bar{\lambda} +1}  K(x,\lambda; y, z) \left( u(z)^{\frac{n+2\sigma}{n-2\sigma}} - u_{x,\lambda}(z)^{\frac{n+2\sigma}{n-2\sigma}} \right) dz \\
&~~~~  + \int_{\bar{\lambda} + 2 \leq |z -x| \leq \bar{\lambda} +3}  K(x,\lambda; y, z) \left( u(z)^{\frac{n+2\sigma}{n-2\sigma}} - u_{x,\lambda}(z)^{\frac{n+2\sigma}{n-2\sigma}} \right) dz \\
& \geq -C \int_{\lambda \leq |z -x| \leq \bar{\lambda} +\varepsilon}  K(x,\lambda; y, z) (|z -x| -\lambda) dz \\
&~~~~ + \int_{\lambda+\varepsilon <  |z -x| \leq \bar{\lambda} +1}  K(x,\lambda; y, z) \left( u_{x,\bar{\lambda}}(z)^{\frac{n+2\sigma}{n-2\sigma}} - u_{x, \lambda}(z)^{\frac{n+2\sigma}{n-2\sigma}} \right) dz\\
&~~~~ + \int_{\bar{\lambda} + 2 \leq |z -x| \leq \bar{\lambda} +3}  K(x,\lambda; y, z) \left( u(z)^{\frac{n+2\sigma}{n-2\sigma}} - u_{x,\lambda}(z)^{\frac{n+2\sigma}{n-2\sigma}} \right) dz,
\endaligned
$$
where in the second inequality we have used
$$
| u(z)^{\frac{n+2\sigma}{n-2\sigma}} - u_{x,\lambda}(z)^{\frac{n+2\sigma}{n-2\sigma}} |\leq C (|z -x| -\lambda). 
$$
Because of \eqref{5CL6-04}, there exists $\delta >0$ such that for all  $\bar{\lambda} \leq  \lambda \leq \bar{\lambda} +\varepsilon$,  we have 
$$
u(z)^{\frac{n+2\sigma}{n-2\sigma}} - u_{x,\lambda}(z)^{\frac{n+2\sigma}{n-2\sigma}} > \delta ~~~~~~\forall ~\bar{\lambda} +2 \leq |z-x| \leq \bar{\lambda} +3,~ z \in \mathbb{R}^n \backslash \mathbb{R}^k. 
 $$
 It is also easy to see that there exists $C>0$ independent of $\varepsilon$ such that for all $\bar{\lambda} \leq  \lambda \leq \bar{\lambda} +\varepsilon$,
 $$
 | u_{x, \bar{\lambda}}(z)^{\frac{n+2\sigma}{n-2\sigma}} - u_{x,\lambda}(z)^{\frac{n+2\sigma}{n-2\sigma}} | \leq C (\lambda - \bar{\lambda}) \leq C\varepsilon ~~~~~~\forall ~ \bar{\lambda} \leq  \lambda  \leq |z-x| \leq \bar{\lambda} +1.  
 $$
Now, by a very similar estimate for the kernel $K$  as in the proof of Theorem \ref{IEthm01},  we can obtain that,  for $\bar{\lambda} \leq \lambda \leq \bar{\lambda} + \varepsilon$ and  for $y\in \mathbb{R}^n \backslash \mathbb{R}^k$ with $\lambda\leq |y -x| \leq \bar{\lambda} +1$,  
$$
u(y) - u_{x,\lambda} (y) \geq (\delta_1\delta_2 c - C \varepsilon^{\frac{2\sigma}{n}}) ( |y -x| -\lambda) \geq 0
$$
if $\varepsilon>0$ is sufficiently small. This and \eqref{5CL6-04} contradict   the definition of $\bar{\lambda}$.  The proof of Claim 2 is completed.   
Thus, we have shown that for every $x\in \mathbb{R}^{n-k} \backslash \{0\}$, 
\begin{equation}\label{SS555}
u_{x,\lambda}(y) \leq u(y)~~~~~~ \forall ~ |y - x| \geq \lambda,~ y \in \mathbb{R}^n \backslash \mathbb{R}^k, ~\forall ~ 0< \lambda< |x|. 
\end{equation}

For  any unit vector $e \in \mathbb{R}^{n-k}$,  for any $a >0$, for any  $\xi=(y, z)\in\mathbb{R}^n$ with $y \in \mathbb{R}^k$ and  $ z\in \mathbb{R}^{n-k}$  satisfying $(z - ae ) \cdot e <0$, and for any $R >a$, we have, by \eqref{SS555} with $x=Re$ and $\lambda = R - a$, 
$$
u(y, z)\geq u_{x,\lambda}(y,z)=\left( \frac{\lambda}{|\xi - x|}\right)^{n-2\sigma} u\left( x + \frac{\lambda^2(\xi -x)}{|\xi -x|^2} \right).
$$
Sending $R$ to infinity in the above,  we obtain 
\begin{equation}\label{SS666}
u(y, z) \geq u(y,  z - 2( z\cdot e  - a )e ).
\end{equation}
Since $z\in\mathbb{R}^{n-k}$ and $a>0$ are arbitrary, \eqref{SS666} gives the radial symmetry of $u$ in the $\mathbb{R}^{n-k}$-variables.  

In particular, if $k=0$, then $u$ is radially symmetric about the origin. Moreover,  \eqref{SS666} also gives
$$
u(z)=u(z_1, z_2, \dots, z_n) \geq u_a(z):=u(2a - z_1, z_2, \dots, z_n)~~~~~ \forall~ z_1\leq a,~ a>0.
$$
This implies that  $u$ is also monotonically decreasing about the origin.
Theorem \ref{IEthm03}  is established. 
\hfill$\square$

\vskip0.10in

\noindent{\it Proof of Theorem \ref{Thm03}. } It follows from Theorem \ref{T-Glo} and Theorem \ref{IEthm03}.  
\hfill$\square$

\vskip0.10in

Finally, we give the proof of Corollary \ref{Cor01}. 

\vskip0.10in

\noindent{\it Proof of Corollary \ref{Cor01}. } The proof is  just a combination of Theorem \ref{T-Glo},  Theorem \ref{IEthm03},  Theorem 5 of \cite{CLO05} and Theorem 1.1 of  \cite{JX19}.   For the reader's convenience, we include the details. By Theorems \ref{T-Glo} and \ref{IEthm03}, we know that $u$ is radially symmetric and monotonically decreasing about the origin and satisfies 
\begin{equation}\label{DY-0900}
u(x)= c_{n,m} \int_{\mathbb{R}^n} \frac{u(y)^{\frac{n+2m}{n-2m}}}{|x -y|^{n-2m}} dy~~~~~\textmd{for} ~  x\in \mathbb{R}^n \backslash \{0\}. 
\end{equation}
Hence,  for any $r>0$ and $\theta \in \mathbb{S}^{n-1}$, we have
$$
\aligned
 u(r \theta) & \geq  c_{n,m} \int_{B_r(0)} \frac{u(y)^{\frac{n+2m}{n-2m}}}{|r\theta -y|^{n-2m}} dy\\
& \geq c_{n,m} u(r)^{\frac{n+2m}{n-2m}} \int_0^r \left( \int_{ \partial B_1(0)}  \frac{1}{|r\theta -s\omega|^{n-2m}} d\omega \right) s^{n-1} ds\\
& = c_{n,m} r^{2m}  u(r)^{\frac{n+2m}{n-2m}} \int_0^1 \left( \int_{ \partial B_1(0)}  \frac{1}{|\theta -t \omega|^{n-2m}} d\omega \right) t^{n-1} dt\\
&= C_1 r^{2m}  u(r)^{\frac{n+2m}{n-2m}} 
\endaligned
$$
for some uniform constant $C_1>0$. It follows that
$$
u(r) \leq C_1 r^{-\frac{n - 2m}{2}}~~~~~~ \textmd{for} ~  \textmd{all} ~   r > 0. 
$$
This proves the upper bound in \eqref{dfgh}.   On the other hand,  by \eqref{DY-0900} we have for any $s=1, \dots, m-1$ that 
\begin{equation}\label{DY-0900-023}
(-\Delta)^s u(x)  = c(n, m, s) \int_{\mathbb{R}^n} \frac{u(y)^{\frac{n+2m}{n-2m}}}{|x -y|^{n-2m+2s}} dy  \geq 0 ~~~~~ \textmd{in} ~ \mathbb{R}^n \backslash \{0\}. 
\end{equation}
Since $0$ is a non-removable singularity,  using Theorem 1.1 of  \cite{JX19} we know there exists $C_2=C_2(n, m, u) >0$ such that 
$$
u(x)\geq C_2  |x|^{-\frac{n - 2m}{2}} ~~~~~~ \textmd{for} ~  x\in \mathbb{R}^n \backslash \{0\}.
$$
This completes the proof of Corollary \ref{Cor01}. 
\hfill$\square$

\vskip0.40in

\noindent{X. Du and H. Yang}\\
Department of Mathematics, The Hong Kong University of Science and Technology\\
Clear Water Bay, Kowloon, Hong Kong\\
E-mail addresses:  xduah@connect.ust.hk (X. Du)~~~~~  mahuiyang@ust.hk (H. Yang)

\end{document}